\documentclass[14pt]{article}
\usepackage[utf8]{inputenc}
\usepackage{pdfpages}
\usepackage{url}
\usepackage{euler} 
\usepackage[toc,page]{appendix}
\usepackage{amssymb, amsthm,amsmath}
\usepackage[Conny]{fncychap}
\usepackage{graphicx}
\usepackage{adjustbox}
\usepackage{comment}
\usepackage{etex}
\usepackage{mathrsfs}
\usepackage[english]{babel}

\usepackage[scaled=0.85]{beramono}%% mono
\usepackage{tikz-cd}
\usepackage{float}
\usepackage{fancyhdr}

\usepackage[left=0cm,right=0cm,top=2cm,bottom=2cm,margin=3cm]{geometry}
\usepackage{amsthm}
\usepackage{bigints}
\usepackage{tikz}
\usepackage[linktoc=all]{hyperref}
\usepackage[capitalise]{cleveref}
\hypersetup{colorlinks=true,linkcolor=blue,citecolor=red}
\usetikzlibrary{matrix,arrows,decorations.pathmorphing}
\usetikzlibrary{patterns,calc,angles,quotes}
\usepackage{tikz-cd}
\usepackage[utf8]{inputenc}
\tikzset{commutative diagrams/.cd}
\newtheorem{theorem}{Theorem}[subsection]

\newtheorem{corollary}[theorem]{Corollary}

\newtheorem{lemma}[theorem]{Lemma}

\newtheorem{proposition}[theorem]{Proposition}

\theoremstyle{definition}
\newtheorem{definition}[theorem]{Definition}

\newtheorem{question}[theorem]{Question} 

\newtheorem{claim}[theorem]{Claim}

\newtheorem{example}[theorem]{Example}

\newtheorem{remark}[theorem]{Remark}

\newtheorem{notation}[theorem]{Notation}

\newcommand{\CrrCopEal}{\op{Corr}(\Ca)^{\otimes}_{\E,\op{all}}}
\newcommand{\CrrCpEal}{\op{Corr}(\Ca)_{\E,\op{all}}}
\newcommand{\CrrCovopEal}{\op{Corr}^{\op{\E-cart}}(\op{Cov}(\Ca))^{\otimes}_{\tilde{\E},\op{all}}}
\newcommand{\CrrCovEopEal}{\op{Corr}^{\op{all-cart}}(\op{Cov}_{\E}(\Ca))^{\otimes}_{\tilde{\E},\op{all}}}
\newcommand{\Ccc}{(\Ca^{\op{op}})^{\coprod,\op{op}}}
\newcommand{\mc}{\mathcal}
\newcommand{\E}{\mathcal{E}}
\newcommand{\D}{\mathcal{D}}

\newcommand{\K}{\mathcal{K}}

\renewcommand{\P}{\mathcal{P}}
\newcommand{\N}{\mathcal{N}}

\newcommand{\op}[1]{\operatorname{#1}}
\renewcommand{\S}{\mathcal{S}}

\newcommand{\C}{\mathcal{C}}

\newcommand{\uhom}{\underline{\op{Hom}}}
\renewcommand{\O}{\mathcal{O}}

\newcommand{\ep}{\epsilon}
\newcommand{\dd}{\delta}
\newcommand{\I}{\mathscr{I}}

\newcommand{\Ca}{\mathcal{C}}

\newcommand{\X}{\mathcal{X}}

\newcommand{\bb}{\bullet}

\numberwithin{subsection}{section}

\newcommand{\J}{\mathcal{J}}

\newcommand{\sset}{\op{Set}_{\Delta}}

\newcommand{\Copt}{\op{\C pt}}

\newcommand{\bx}{\square}

\newcommand{\Kpt}{\op{\K pt}}

\newcommand{\Y}{\mathcal{Y}}
\fancypagestyle{main}{\fancyhf{}
	\fancyhead[LE,RO]{\scriptsize Enh op map for six fun formalism}
	\fancyhead[RE,LO]{\scriptsize \rightmark }
	\fancyfoot[CE,CO]{\thepage}}

\title{Six-Functor Formalisms III : The construction and extension of 6FFs}
\author{Chirantan Chowdhury}
\date{\today}

\begin{document}

\maketitle{}
\begin{abstract}
     This article is the last of the series of articles where we reprove the foundational ideas of abstract six-functor formalisms developed by Liu-Zheng. We prove the theorem of partial adjoints, which is a simplicial technique of encoding various functors altogether by taking adjoints along specific directions. Combined with the $\infty$-categorical compactification theorem from the previous article, we can construct abstract six-functor formalisms in reasonable geometric setups of our interest.  We also reprove the simplified versions of the DESCENT program due to Liu-Zheng, which allows us to extend such formalisms from smaller to larger geometric setups.
\end{abstract}
\tableofcontents
\section{Introduction}

The six-functor formalism was developed by Grothendieck and many others to understand duality in the context of \'etale cohomology of schemes, which resulted in solving Weil conjectures over finite fields. Briefly, a six-functor formalism  consists of six-functors $f^*,f_*,f_!,f^!, \uhom(-,-)$ and $\otimes$ assoicated to coefficient systems like \'etale cohomology, $\D$-modules etc. Although the classical formalism is formulated in the language of triangulated categories (\cite{Cisinski_2019}), the modern language of higher category theory developed by Lurie (\cite{HTT},\cite{HA} and \cite{SAG}) has upgraded such notions in this modern language known as abstract six-functor formalism.  \footnote{In these articles, we only consider the abstract six-functor formalism using the language of $\infty$-categories due to Lurie. Abstract six-functor formalisms have been studied using the language of derivators for example: \cite{hörmann2022derivator}}The $(\infty,1)$-categorical formalism of such notions relies on unpublished works of Liu-Zheng (\cite{Gluerestnerv} and \cite{liu2017enhanced}).\footnote{There is another formulation of abstract six-functor formalisms using the language of $(\infty,2)$-categories developed by Gaitsgory-Rozenblyum }. This article is the last of the series of articles where we reprove the foundational ideas of abstract six-functor formalisms developed by Liu-Zheng. We prove the theorem of partial adjoints, which is a simplicial technique of encoding various functors altogether by taking adjoints along specific directions. Combined with the $\infty$-categorical compactification theorem from the previous article, we can construct abstract six-functor formalisms in reasonable geometric setups of our interest. We also reprove the simplified versions of the DESCENT program due to Liu-Zheng, which allows us to extend such formalisms from smaller to larger geometric setups.\\

Before delving into the main results of this article, let us briefly restate what we did in the last two articles of this series. In \cite{chowdhury2023sixfunctorformalismsi}, we prove a technical theorem (\cref{maintechnicalsimplicesthm}) which helps us to solve various lifting problems involved in the realm of abstract six-functor formalisms. This uses the model structure of marked-simplicial sets and the category of simplices. The second article (\cite{chowdhury2024sixfunctorformalismsii}) relies on reproving the $\infty$-categorical version of Deligne's compactification, which is a higher analog of defining exceptional pushforward functors by gluing functors along open and proper morphisms (\cref{compthm}). This uses the technical theorem of the first article crucially. We also introduce the language of multi simplicial sets and other combinatorial simplicial sets related to compactifications and decomposing commutative squares into pullback squares. \\

Let $(\Ca,\E)$ be a marked $\infty$-category where $\Ca$ admits finite products $\E$ contains isomorphisms and is stable under pullbacks and compositions (such a pair is called a geometric setup), Liu-Zheng and Mann define the $(\infty,1)$-category of correspondences denoted by $\CrrCpEal$. It is an $\infty$-category where objects are objects of $\Ca$.
A $1$-simplex i.e an edge between $X_0$ and $X_1$ where $X_0,X_1 \in \Ca$ is a diagram of the form :
         \begin{equation*}
             \begin{tikzcd}
                 X_0 & X_{01}\arrow[l] \arrow[d,"f"] \\
                 {} & X_1
             \end{tikzcd}
         \end{equation*}
         where $f \in E$.
The $\infty$-category $\CrrCpEal$ comes equipped with two morphisms $\pi_{\op{all}}:\Ca^{op} \to  \CrrCpEal$ and $\pi_{\E}: \Ca_{\E} \to \CrrCpEal$ where $\Ca_{\E}$ is the full subcategory of $\Ca$ spanned by edges in $\E$.

\begin{definition}\cite[Definition A.5.6]{padic6functorlucasmann}
An abstract $3$-functor formalism is a lax symmetric monoidal functor :
\begin{equation}
    \D_{(\Ca,\E)} : \CrrCpEal \to \op{Cat}_{\infty}
 \end{equation}
    
\end{definition}     
Given an abstract $3$-functor formalism $\D_{(\Ca,\E)}$ precomposing with $\pi_{\op{all}}$ and $\pi_{\E}$ gives us the functors : 
\begin{equation}
    \D^* : \Ca^{\op{op}} \to \op{Cat}_{\infty} \quad;\quad \D_! : \Ca_{\E} \to \op{Cat}_{\infty}
 \end{equation}
 These are the pullback and exceptional pushforward functors. Combined with symmetric monoidiality, which gives the tensor structure on $\D(X):= \D_{(\Ca,\E)}(X)$, we have three functors, hence the name $3$-functor formalism. For any $f: Y \to X$ in $\Ca$ ($\E$), let $f^*:=\D^*(f)$ ($f_!=\D_!(f)$).

 \begin{definition}\cite[Definition A.5.7]{padic6functorlucasmann}
An abstract  $6$-functor formalism is a an abstract $3$-functor formalism such that $f^*,f_!$ and $\otimes$ have adjoints $f_*,f^!,\uhom(-,-)$ repsectively.
 \end{definition}

 The natural question arises given $(\Ca,\E)$ under what assumptions one can construct a $3$-functor formalism. Motivated by Deligne's gluing techniques, we have the following theorem.

 \begin{theorem}[Simplified version of \cref{mainconstructiontheorem}]\label{mainconstructiontheorembrief}
     Let $(\Ca,\E)$ be pair as above. Let $(\I,\P)$ be two subsets of edges with the following assumptions : \begin{enumerate}
         \item Both $\I$ and $\P$ are stable under pullbacks and compositions and contain isomorphisms.
         \item Every morphism $f \in \E$ admits a decomposition $f =\bar{f} \circ j$ where $j \in \I$ and $\bar{f} \in \P$.
         \item  Given $f : X \to Y$ in $\Ca$ and $g: Y \to Z$ in $\I$($\P$) then $f \in \I$($\P$) iff $g \circ f \in \I$($\P$).
         \item Every morphism $f \in \I \cap \P$ is $k$-truncated for some $k \ge -2$. 
         \end{enumerate}
         Let 
         \begin{equation}
             \D^{*\otimes} : \Ca^{\op{op}} \to \op{CAlg}(\op{Cat}_{\infty})
         \end{equation}
         with the following assumptions :
\begin{enumerate}
     \item For every $X \in \Ca$, $\D(X)$ is closed.
     \item For every $f :X \to Y$ $f^*$ admits a right adjoint $f_*$.
     \item For every $f :X \to Y$ in $\I$, $f^*$ admits a left adjoint $f_{\#}$ which satisfies $\I$-projection formula and $\I$-base change.
    \item For every $f: X \to Y$ in $\P$, $f_*$ satisfies $\P$-projection formula and $\P$-base change. Also, $f_*$ admits a right adjoint $f^!$.    
    \item For every $f: X \to Y$ in $\I$, $f_{\#}$ satisfies base change with respect to $(-)_*$.
    \end{enumerate}
    Then $\D^{*\otimes}$ can be extended to an abstract $6$-functor formalism:
    \begin{equation}
        \D_{(\Ca,\E)} : \CrrCpEal \to \op{Cat}_{\infty}
     \end{equation}
     such that $f \in I$, we have $f_!=f_{\#}$ and for $f \in \P$, we have $f_! =f_*$.      
  \end{theorem}
The above theorem uses partial adjoints, $\infty$-categorical compactification, and relating the $\infty$-category of correspondences with multisimplicial sets. A similar result shall be proved in a model-independent way using the language of $(\infty,2)$-categories in a joint work in progress by Mann, Heyer, and Perutka (see \cref{alternateconstructionremarkMann} for more details).

The main advantage of such an abstract six-functor formalism is to extend them from smaller geometric setups like schemes, diamonds, etc, to larger geometric setups like algebraic stacks, v-stacks, etc. Applications of such extensions are used in the context of arithmetic geometry (\cite{padic6functorlucasmann}) and motivic homotopy theory (\cite{khan2021generalized}, \cite{Chowdhury}). To state the theorems, we introduce generalized notions of pairs of geometric setups mimicking the settings of schemes embedded in algebraic stacks.

\begin{definition}
\begin{enumerate}
    \item An inclusion of two marked $\infty$-categories $(\Ca,\S,\E) \subset (\Ca',\S',\E')$ is a \textit{nice geometric pair} if the following conditions hold :
    \begin{enumerate}
        \item Each of the four pairs $(\Ca,\S),(\Ca,\E),(\Ca,\S')$ and $(\Ca',\E')$ are geometric setups.
        \item $\S' \cap \Ca_1 =\S$.
        \item For  $X' \in \Ca'$, there exists a morphism $x : X \to X'$ called an \textit{atlas} such that $X \in \Ca$ and for every $Y \to X'$ where $Y \in \Ca$, the base change $Y \times_{X'} X \to Y$ lies in $S$.
        \item For every $f : X' \to Y'$ in $\E'$ and for every atlas $y : Y \to Y'$, the base change morphism $X' \times_{Y'}Y \to Y$ is in $\E.$
    \end{enumerate}
    \item An inclusion of $2$-marked $\infty$-categories $(\Ca,\S,\E) \subset (\Ca,\S,\E')$ is an \textit{exceptional pair} if the following conditions are satisfied :
    \begin{enumerate}
        \item The pairs $(\C,\S),(\Ca,\E),(\Ca,\E')$ are geometric setups.
        \item $\S \subset \E$.
        \item For every $f: X \to Y$ in $\E'$, there exists a morphism of augmented simplicial objects 
    \begin{equation*}
        f_{\bb} : X_{\bb} \to Y_{\bb}
        \end{equation*}
        where 
        \begin{itemize}
            \item $f_{-1}=f$
            \item $f_n \in E$ for $n \ge 0$,
            \item $X_{\bb} \to X$ and $Y_{\bb} \to Y$ are $\S$-hypercovers.
        \end{itemize}
        
    \end{enumerate}
    \end{enumerate}
\end{definition}

Nice geometric pairs allow us to extend our abstract six-functor formalisms to bigger setups, while exceptional pairs allow us to extend the exceptional functors to a larger class of morphisms. The theorem of extending abstract six-functor formalisms can be stated as follows:
\begin{theorem}\label{extendinssixfunctorCEside}[Combination of \cref{extendingsixfunctorCside} and \cref{extendingsixfunctorEside}]
    \begin{enumerate}
        \item Let $(\Ca,\S,\E) \subset (\Ca,\S',\E')$ be a nice geometric pair and let 
        \begin{equation}
            \D_{(\Ca,\E)} : \CrrCpEal \to \op{Cat}_{\infty}
        \end{equation}
        be an abstract six-functor formalism with the property that $\D^*$ satisfies descent along $\S$-\v{C}ech covers. Then $\D_{(\Ca,\E)}$ can be extended to an abstract six-functor formalism:
        \begin{equation}
            \D_{(\Ca',\E')} : \op{Corr}(\Ca')_{\E',\op{all}} \to \op{Cat}_{\infty}
        \end{equation}
        \item Let $(\Ca,\S,\E) \subset (\Ca,\S,\E')$ be an exceptional pair and let 
        \begin{equation}
            \D_{(\Ca,\E)} : \CrrCpEal \to \op{Cat}_{\infty}
        \end{equation}
        be an abstract six-functor formalism with the property that $\D_!$ satisfies codescent along $\S$-hypercovers. Then $\D_{(\Ca,\E)}$ can be extended to an abstract six-functor formalism:
        \begin{equation}
            \D_{(\Ca,\E')} : \op{Corr}(\Ca)_{\E',\op{all}} \to \op{Cat}_{\infty}
        \end{equation}
    \end{enumerate}
\end{theorem}

The above theorem is a simplified version of the DESCENT program developed by Liu-Zheng (\cite{liu2017enhanced}), an algorithm to extend abstract six-functor formalisms and other properties. The proof uses the theory of Dwyer-Kan localizations. Similar extension results are also available in \cite[Appendix A.5]{padic6functorlucasmann}. \\

We now briefly outline the main sections of this article :
\begin{enumerate}
    \item In Section 2, we briefly recall the main results from the previous two articles, which involve stating \cref{maintechnicalsimplicesthm} and the $\infty$-categorical compactification (\cref{compthm}).
    \item In Section 3, we prove the theorem of partial adjoints. It is a technical theorem that allows us to encode various functors simultaneously by taking adjoints along specific morphisms. This plays a crucial tool in constructing the abstract six-functor formalisms. We briefly recall the notion of left and right adjointable squares due to Lurie, followed by stating and proving the partial adjoints theorem (\cref{prtadj}).
    \item In Section 4, we recall the notion of $\infty$-category of correspondences and prove its main properties, which includes establishing the symmetric monoidal structure of $\infty$-category of correspondences. This section has already appeared in the Appendix of joint work of the author with Alessandro D'Angaelo in \cite{chowdhury2024nonrepresentablesixfunctorformalisms}. The section discusses the $\infty$-category of correspondences with the language of bisimplicial sets (\cref{correspondencebisimplicialsetsequivalence.}). It serves as the last step in proving \cref{mainconstructiontheorembrief}.
    \item In Section 5, we recall abstract 3 and 6-functor formalisms due to Mann and prove \cref{mainconstructiontheorembrief}.
    \item The last section involves proving the extension of six-functor formalisms for nice geometric and exceptional setups (\cref{extendinssixfunctorCEside}). These results have also appeared in the Appendix of \cite{chowdhury2024nonrepresentablesixfunctorformalisms}.
    \item In the Appendix, we briefly recall the notion of Dwyer-Kan localizations and prove an essential criterion for Dwyer-Kan localizations (\cref{localizationcriterion}). This serves as a key input for proving \cref{extendinssixfunctorCEside}.
\end{enumerate}

\subsection*{Acknowledgements:}
The paper was written while the author was a PostDoc under Prof.Dr. Timo Richarz at University of TU Darmstadt. C.Chowdhury acknowledges support (through Timo Richarz) by the European Research Council (ERC) under Horizon Europe (grant agreement nº 101040935), by the Deutsche Forschungsgemeinschaft (DFG, German Research Foundation) TRR 326 \textit{Geometry and Arithmetic of Uniformized Structures}, project number 444845124 and the LOEWE professorship in Algebra, project number LOEWE/4b//519/05/01.002(0004)/87. \\

The author would like to thank Alessandro D'Angelo for discussing technical details regarding the category of correspondences and abstract six-functor formalisms. He would also like to thank R{\i}zacan \c{C}ilo\u{g}lu for helpful discussions regarding the paper.

\subsection*{Conventions}
The paper relies on notations and definitions from the papers \cite{Gluerestnerv} and \cite{liu2017enhanced}. We shall omit referencing the documents as it will be implicit throughout the article. We also encourage the reader to look into the previous articles \cite{chowdhury2023sixfunctorformalismsi} and \cite{chowdhury2024sixfunctorformalismsii} for a detailed understanding of the recollection of the results. The results in all of the three articles combined is self-contained as possible without referencing \cite{Gluerestnerv} and \cite{liu2017enhanced}.\\
Lastly, we also freely use the language of $\infty$-categories developed by Lurie in \cite{HTT}, \cite{HA} and \cite{SAG}. 
\section{Recollection of notations and results.}
In this section, we recall the relevant notions and from \cite{chowdhury2023sixfunctorformalismsi}.and \cite{chowdhury2024sixfunctorformalismsii}

\subsection{Constructing functors using category of simplices..}
\begin{definition}
    Let $\J$ be a (small) ordinary category. Let $(\sset)^{\J}$ be the category where objects are functors from $\J \to \sset$ and morphisms are natural transformations.
\end{definition}
We now introduce the constant and global section functor related to $(\sset)^{\J}$.
\begin{notation}
    For every simplicial set $X$, we have the constant simplicial set functor $c(X):=X_{\J}$ defined by sending any object $j$ to the simplicial set $X$. The association is functorial and thus we have a functor :
     \[ c: \sset \to (\sset)^{\J}\]
    
\end{notation}

\begin{definition}\label{globalsectionfunctor}
We define the \textit{global section functor} 
\[ \Gamma : (\sset)^{\J} \to \sset\]
as follows : 
\[ \Gamma(F) =(\Gamma(F)_n := \op{Hom}_{(\sset)^{\J})}(\Delta^n_{\J},F))_n.\]
\end{definition}
\begin{example}
Let $F$ be the constant functor $c(X)$ where $X \in \sset$. Let us compute $\Gamma(F)$.
The $n$-simplices of $\Gamma(F)$ are given by the set of natural transformations from $\Delta^n_J \to c(X)$. Every such natural transformation is equivalent to give a single map $\Delta^n \to X$. In particular the $n$-simplices of $\Gamma(F)$ are given by $n$-simplices of $X$.Thus $\Gamma(F)= X$. In particular, we prove that $\Gamma \circ c = \op{id}_{\sset}$.

\end{example}
\begin{remark}
Recall from classical category theory, given a complete category $\Ca$ and an small category $I$, we have the pair of adjoint functors: 
 \[ c: \Ca \leftrightarrows \op{Fun}(I,\Ca) : \op{lim}\]
 where $\op{lim}$ is the functor which takes an object which is a functor $F : I \to \Ca$ to its limit $\op{lim}(F) \in \Ca$. \\

 Let $\Ca= \sset$ and $I=\J$. As the category of simplicial sets is complete,  we see that 
 \[ \Gamma = \op{lim}. \]
 In the other words, the global section functor is the limit functor which takes every functor to its limit in the category of simplicial sets.
 
\end{remark}

\begin{definition}\label{catofsimpldef}
    Let $K$ be a simplicial set. Then the \textit{category of simplicies over $K$} is a category consisting of :
    \begin{enumerate}
        \item Objects : $(n,\sigma)$ where $n \ge 0$ and $\sigma \in K_n$.
        \item Morphisms: $ p:(n,\sigma)\to (m,\sigma')$ is a morphism $p: [n] \to [m]$ such that $p(\sigma)=\sigma'$.
    \end{enumerate}
\end{definition}

The relevant functor associated to the category of simplices is the mapping functor. 
\begin{definition}\label{mappingfunctor}
    Let $K$ be a simplicial set and $\Ca$ be a $\infty$-category. The \textit{mapping functor} \[ \op{Map}[K,\Ca] : (\Delta_{/K})^{op} \to \sset \] is defined as follows: 
    \[ (n,\sigma) \to \op{Map}^{\sharp}((\Delta^n)^{\flat},\Ca^{\natural}) \cong \op{Fun}^{\cong}(\Delta^n,\Ca). \]
     Here $\op{Map}^{\sharp}((\Delta^n)^{\flat},\Ca^{\natural})$ is the internal mapping space in the category of marked simplicial sets. It is the largest Kan complex contained in $\op{Fun}(\Delta^n,\Ca)$.
 \end{definition}

 \begin{remark}

       For a simplicial set $K$  and an $\infty$-category $\Ca$, we have the following equality of simplicial sets: 
     \[ \Gamma(\op{Map}[K,\Ca]) = \op{Map}^{\sharp}(K^{\flat},\Ca). \] This is proved in \cite[Lemma 3.3.3]{chowdhury2023sixfunctorformalismsi}

 \end{remark}

 We recall the main theorem from \cite{chowdhury2023sixfunctorformalismsi}. 
 \begin{theorem}\label{maintechnicalsimplicesthm}\cite[Theorem 4.1.1]{chowdhury2023sixfunctorformalismsi}
Let $K',K$ be simplicial sets and $\Ca$ be a $\infty$-category. Let $f': K' \to \Ca$ and $i:K' \to K$ be morphisms of simplicial sets. Let $\N \in (\sset)^{(\Delta_{/K})^{op}}$ and $\alpha: \N \to \op{Map}[K,\Ca]$ be a natural transformation. If
\begin{enumerate}
 \item (\textit{Weakly contractibility}) for $(n,\sigma) \in \Delta_{/K}$, $\N(n,\sigma)$ is weakly contractible,
 \item (\textit{Compatability with $f'$}) there exists $\omega \in \Gamma(i^*\N)_0$ such that $\Gamma(i^*\alpha)(\omega)=f'$,
 \end{enumerate}
 
then there exists a map $f: K \to \Ca$ such that the following diagram 
\begin{equation}
    \begin{tikzcd}
        K' \arrow[r,"f'"] \arrow[d,"i"] & \Ca \\
        K \arrow[ur,"f"] & {}
    \end{tikzcd}
\end{equation}
 commutes. In other words, $f' \cong f \circ i $ in $\op{Fun}(K',\Ca)$.
\end{theorem}
\begin{remark}
    The above theorem has been a key technical tool in the formalism of abstract six-functor formalisms. This theorem has been useful in proving the $\infty$-categorial compactification which is the content of \cite{chowdhury2024sixfunctorformalismsii} and has been recalled in the following subsection. \\

    Putting $K' = \phi$ the empty simplicial set also provides a way of constructing morphisms between simplicial sets. This shall be used in the context of partial adjoints (\cref{prtadj}).
\end{remark}

\subsection{The $\infty$-categorical compactification.}
Let $\Ca$ be an $\infty$-category and let $(\E_1,\E_2)$ be a pair of collection of edges in $\Ca$ satisfying some nice conditions (see \cref{compthm} for more details). Let us consider a new simplicial set $\dd^*_2 \Ca^{\op{cart}}_{\E_1,\E_2}$. The $n$-simplices of $\dd^*_2\Ca^{\op{cart}}_{\E_1,\E_2}$ are $n \times n$ grids of the form 
	
	\begin{equation}
		\begin{tikzcd}
			X_{00} \arrow[r] \arrow[d] & X_{01} \arrow[r] \arrow[d] & \cdots & X_{0n} \arrow[d] \\
			X_{10} \arrow[r]  & X_{11} \arrow[r] & \cdots & X_{1n} \\
			\vdots & \vdots & \vdots & \vdots \\
			X_{n1} \arrow[r] & X_{n2} & \cdots & X_{nn}
		\end{tikzcd}
	\end{equation}
	where vertical arrows are in $\E_1$, horizontal arrows are in $\E_2$ and each square is a pullback square. Also one has a natural morphism $p:\dd^*_2\Ca^{\op{cart}}_{\E_1,\E_2} \to \Ca$ induced by composition along the diagonal. Before stating the main theorem, let us  recall the notion of admissible edges.
	\begin{definition}
		Let $\Ca$ be an $\infty$-category. Let $\E$ be a collection of morphisms in $\Ca$. Then $\E$ is said to be \textit{admissible} if 
		\begin{enumerate}
			\item $\E$ contains every identity morphism in $\Ca$.
			\item $\E$ is stable under pullbacks.
			\item For every pair of composable morphisms $p \in \E$ an $q$ a morphism in $\Ca$, then if $p \circ q \in \E $ implies $ q \in \E$. 
		\end{enumerate}
	\end{definition}
	
	\begin{theorem}\label{compthm}
		Let $\Ca$ be an $\infty$-category and $\E_1,\E_2$ be a collection of edges in $\Ca$ with the following conditions:
  \begin{enumerate}
      \item For every morphism in $f \in \Ca$, there exists a $2$-simplex in $\Ca$ of the form :
      \begin{equation}
          \begin{tikzcd}
              {} & y \arrow[dr,"p"] & {}\\
              x \arrow[ur,"q"] \arrow[rr,"f"] && z
              \end{tikzcd}
              \end{equation}
              where $p \in \E_1$ and $q \in \E_2$.
\item Every morphism $f \in \E_1\cap \E_2$ is $k$-truncated for $k \ge -2$.
 
\item The edges $\E_1$ and $\E_2$ are admissible.

  \end{enumerate}
  Then for any $\infty$-category $\D$, there exists a solution to the lifting problem:
\begin{equation}
\begin{tikzcd}
    \dd^*_2\Ca^{\op{cart}}_{\E_1,\E_2} \arrow[d,"p"] \arrow[r,"g"] &\D\\
    \Ca \arrow[ur,"g'",dotted] & {}.
    \end{tikzcd}
\end{equation}
  
	\end{theorem}

\begin{remark}
    The above theorem is the $\infty$-categorical version of Deligne's compactifications while constructing exceptional pushforward functors in the context of \'etale cohomology of schemes.
\end{remark}

We state now a more general version of the theorem which is relevant for this paper.

\begin{theorem}\label{compthmalledge}
    		Let $(\Ca,\E)$ be a marked $\infty$-category and $\E_1,\E_2$ be a collection of edges in $\Ca$ with the following conditions:
  \begin{enumerate}
      \item For every morphism in $f \in \E$, there exists a $2$-simplex in $\Ca$ of the form :
      \begin{equation}
          \begin{tikzcd}
              {} & y \arrow[dr,"p"] & {}\\
              x \arrow[ur,"q"] \arrow[rr,"f"] && z
              \end{tikzcd}
              \end{equation}
              where $p \in \E_1$ and $q \in \E_2$.
\item Every morphism $f \in \E_1\cap \E_2$ is $k$-truncated for $k \ge -2$.
 
\item The edges $\E_1,\E_2$ and $\E$ are admissible.

  \end{enumerate}
  Then for any $\infty$-category $\D$, there exists a solution to the lifting problem:
\begin{equation}
\begin{tikzcd}
    \dd^*_3\Ca^{\op{cart}}_{\E_1,\E_2,\op{ALL}} \arrow[d,"p"] \arrow[r,"g"] &\D\\
    \dd^*_2\Ca_{\E,\op{all}} \arrow[ur,"g'",dotted] & {}.
    \end{tikzcd}
\end{equation}
\end{theorem}
\begin{remark}
    The proof of the above theorem follows verbatim as the proof of \cref{compthm}. We need to keep track of an additional edge here and follow the proofs of the relevant statements with an additional direction. In order to do this, one needs the definition of simplicial set of compactifications to have compactifications only for edges in $\E$. 
    
    \end{remark}
\section{Partial adjoints.}
In this section, we prove the theorem of partial adjoints. It is the last key tool; that we need in order to construct six-functor formalisms. The idea of abstract six-functor formalism is to define a functor which must encode the six-functors in a compatible way. In particular we need to encode the pullback functors and exceptional functors alltogether so that we get base change. As base changes include cartesian squares, it is natural to use the language of multisimplicial sets. Let us provide a short intro on the importance of partial adjoints.\\

Let $(\Ca,\E_1,\E_2)$ be a two marked $\infty$-category. Consider the simplicial set $\dd^*_2\Ca^{\op{cart}}_{\E_1,\E_2}$. The $n$-simplices of $\dd^*_2\Ca^{\op{cart}}_{\E_1,\E_2}$ are $n \times n$ grids of the form 
	
	\begin{equation}
		\begin{tikzcd}
			X_{00} \arrow[r] \arrow[d] & X_{01} \arrow[r] \arrow[d] & \cdots & X_{0n} \arrow[d] \\
			X_{10} \arrow[r]  & X_{11} \arrow[r] & \cdots & X_{1n} \\
			\vdots & \vdots & \vdots & \vdots \\
			X_{n1} \arrow[r] & X_{n2} & \cdots & X_{nn}
		\end{tikzcd}
	\end{equation}
	where vertical arrows are in $\E_1$, horizontal arrows are in $\E_2$ and each square is a pullback square. Also one has a natural morphism $p:\dd^*_2\Ca^{\op{cart}}_{\E_1,\E_2} \to \Ca$ induced by composition along the diagonal.\\
    
    Suppose we have a functor $F : \Ca^{\op{op}}\to \op{Cat}_{\infty}$. Precomposing with $p$ gives the functor :
    \begin{equation}
        G: \dd^*_{2,\{2\}}\Ca_{\E_1,\E_2} \to \op{Cat}_{\infty}
    \end{equation}
    Suppose we have the following property that for every $f: Y \to X$ in $\E_2$, $f^*:=F(f)$ admits a right adjoint $f_* : F(Y) \to F(X)$ which satisfies base change property with respect to morphisms in $\E_1$. Then the question arises the following : 
    \begin{question}
        Is it possible to define a functor $G': \dd^*_{2,\{1\}}\Ca^{\op{cart}}_{\E_1,\E_2} \to \op{Cat}_{\infty}$ such that 
        \begin{enumerate}
            \item $G'_{\E_1} : \Ca_{\E_1}^{\op{op}} \to \dd^*_{2,\{1\}}\Ca^{\op{cart}}_{\E_1,\E_2}\to \op{Cat}_{\infty}$ is $F|_{\Ca_{\E_1}}$.
    
    \item     $G'_{\E_2} : \Ca_{\E_2} \to \dd^*_{2,\{1\}}\Ca^{\op{cart}}_{\E_1,\E_2} \to \op{Cat}_{\infty}$ sends $f \mapsto f_*$.
    \end{enumerate}
    \end{question}
    The above questions asks about the existence of $G'$ which encodes both the functors $f^*$ and $f_*$ in a coherent way which is a key feature in formulation of six-functor formalism. The theorem of partial adjoints answeres this question. \\
In order to state the theorem of partial adjoints, we introduce the notion of right and left adjointable squares and its important properties.
\subsection{Adjointable squares.}
	\begin{definition}\label{adjointablesquares}
		Suppose we have a diagram of $\infty$-categories 
		\begin{equation}
			\sigma:=
			\begin{tikzcd}
				\Ca \arrow[r,"G"] \arrow[d,"U"] & \D \arrow[d,"V"] \\
				\Ca' \arrow[r,"G'"] & \D'
			\end{tikzcd}
		\end{equation}
		which commutes upto specified equivalence  
		\[ \alpha: V \circ G \cong G' \circ U. \]
		We say $\sigma$ is \textit{left adjointable} if $G$ and $G'$ admit left adjoints $F$ and $F'$ respectively and if the composite transformation 
		\[ F' \circ V \to F' \circ V \circ G \circ F \cong F' \circ G' \circ U \circ F \to U \circ F \] is an equivalence.
	\end{definition}
	
	\begin{remark}
		Some remarks on the definition above:
		\begin{enumerate}
			\item 	We have the dual notion of \textit{right adjointable squares}.
			
			\item  The notion of left and right adjointable squares in classical category theory is called the \textit{Beck-Chevalley condition}.
		\end{enumerate}
		
	\end{remark}

\begin{definition}
		Let $S$ be a simplicial set, we define subcategories
		
		\[ \op{Fun}^{LAd}(S,\op{Cat}_{\infty}), \op{Fun}^{RAd}(S,\op{Cat}_{\infty}) \subset \op{Fun}(S,\op{Cat}_{\infty}) \] as follows:
		\begin{enumerate}
			\item Let $F \in \op{Fun}(S,\op{Cat}_{\infty})$. Then $F \in \op{Fun}^{LAd}(S,\op{Cat}_{\infty}) (\op{Fun}^{RAd}(S,\op{Cat}_{\infty}))$ if and only if for every edge $s \in s'$ $F(s) \to F(s')$ admits a left (right) adjoint.
			\item Let $\alpha: F \to F'$ be a morphism in $\op{Fun}(S,\op{Cat}_{\infty})$ where $F, F' \in \op{Fun}^{LAd}(S,\op{Cat}_{\infty}) \linebreak(\op{Fun}^{RAd}(S,\op{Cat}_{\infty}))$, then $\alpha$ is a morphism in $\op{Fun}^{LAd}(S,\op{Cat}_{\infty}) (\op{Fun}^{RAd}(S,\op{Cat}_{\infty}))$ if for every $s \in s'$, the diagram 
			\begin{equation}
				\begin{tikzcd}
					F(s) \arrow[r] \arrow[d] & F(s') \arrow[d] \\
					F'(s) \arrow[r] & F'(s')
				\end{tikzcd}
			\end{equation}
			is left(right) adjointable.
		\end{enumerate}
	\end{definition}
	
	We have the following important corollary. 
	
	\begin{corollary}(Corollary 4.7.4.18 of \cite{HTT})\label{adjsquarethm}
		\begin{enumerate}
			\item The $\infty$-categories $\op{Fun}^{LAd}(S,\op{Cat}_{\infty})$ and $\op{Fun}^{RAd}(S,\op{Cat}_{\infty})$ are presentable and in particular, they admit all small limits.
			\item There is a canonical equivalence of $\infty$-categories 
			\[ \op{Fun}^{LAd}(S^{op},\op{Cat}_{\infty}) \cong \op{Fun}^{RAd}(S,\op{Cat}_{\infty}). \]
		\end{enumerate}
	\end{corollary}

    \subsection{Statement and proof of partial adjoints.}
We now state the theorem of partial adjoints (see also \cite[Proposition 1.4.4]{liu2017enhanced}). We freely use the notations of multisimplicial sets defined in \cite[Section 3]{chowdhury2024sixfunctorformalismsii}.
\begin{theorem}\label{prtadj}
		Consider quadruples $(I,J,R,F)$ where $J \subset I$ are finite sets, $R$ an $I$-simplicial set and $ f: \dd^*_I R \to \op{Cat}_{\infty}$ a functor satisfying the following conditions:
		\begin{enumerate}
			\item For every $j \in J$ and for every edge $ e \in \ep^I_j(R)$, $F(e)$ has  a right adjoint.
			\item For every $ i \in J^c:=I/J$ and $ j \in J$, every square $\tau \in (\ep^I_{i,j}R)_{1,1}$, the commutative square $F(\tau)$ is right adjointable.
			\end{enumerate}
		Then, there exists a functor $ F_J: \dd^*_{I,J}R \to \op{Cat}_{\infty}$ satisfying the following conditions:
		\begin{enumerate}
			\item $F$ and $F_J$ are the same functors on the sub-simplicial set $\dd^*_{J^c}(\Delta_{i_c})_*R \subset \dd^*_IR,\dd^*_{I,J}R$. Here $ i_c: J^c \subset I$.
			\item For every $j \in J$ and for every edge $e \in \ep^I_j(R)$, $F_J(e)$ is a right adjoint to $F(e)$.
			\item For every $ i \in J^c, j \in J$ any every square $\tau \in (\ep^I_{i,j}R)_{1,1}$, $F_J(\tau)$ is right adjointable square to $F(\tau)$.
		\end{enumerate}
	\end{theorem}
    \begin{remark}
    Let $R$ be an $I$-simplicial set. We briefly recall the meaning of the conditions in the theorem:
        \begin{enumerate}
            \item Let $i \in I$. For a $1$-simplex in $\ep^I_iR$, this corresponds to \[\Delta^{0,\cdots,1,\cdots, 0} \to R\] where $1$ is in the $i$-th position. Applying $\dd^*_I$ functor yields a morphism :
            \begin{equation}
                \Delta^1 \to \dd^*_IR
            \end{equation}
\item Let $i,i' \in I$, A $(1,1)$-simplex of $\ep^I_{i,i'}R$ is a morphism of the form \[\Delta^{0,\cdots,1,\cdots,1,\cdots 0} \to R \]   where   $1$ is in the $i$ and $i'$th position. Applying $\dd^*_I$ yields a morphism of the form
\begin{equation}
    \Delta^1 \times \Delta^1 \to \dd^*_iR
\end{equation}
which is a commutative square in this simplicial set. 
\end{enumerate}
    \end{remark}

\begin{proof}
    The idea is to use the technical theorem of constructing functors using the category of simplicies. We verify the conditions and setup of \cref{maintechnicalsimplicesthm} as follows: 
    \begin{enumerate}
    \item \textbf{Defining $\N(-)$:}
    Let $ \tau : \Delta^n \to \dd^*_{I,J}R$. This corresponds to a morphiasm of the form :
    \begin{equation}
        \Delta^{\underline{n}} \to \op{op}^I_JR.
    \end{equation}
    Let $\N(\tau):= (N(\Delta_{/\dd^*_{I,J}R}))_{\tau/}$. As this has an initial object this is weakly contractible. As the category of simplicies is functorial, we have a functor :
    \begin{equation}
        \N: (\Delta_{/\dd^*_{I,J}R})^{op} \to \sset 
    \end{equation}
\item \textbf{Defining $\alpha:$} Unravelling the definition of $\op{op}^I_J$ and applying the functor $\dd^*_I$, we get that $\tau$ yields a morphism of the form 
\begin{equation}
    (\prod _{j \in J} \Delta^n)^{\op{op}} \times (\prod_{i in \in J^c} \Delta^n) \to \dd^*_IR
 \end{equation}
 Applying $F$, we get a morphism
 \begin{equation}
     F_{\tau} :     (\prod _{j \in J} \Delta^n)^{\op{op}} \times (\prod_{i  \in J^c} \Delta^n) \to \op{Cat}_{\infty}
 \end{equation}
 which using the Hom and product adjunction, can be realized as a morphism :
 \begin{equation}
     F_{\tau} :  (\prod_{i \in J^c} \Delta^n) \to \op{Fun}(    (\prod _{j \in J} \Delta^n)^{\op{op}},\op{Cat}_{\infty}).
 \end{equation}
 The conditions of the theorem are equivalent to the following : 
 \begin{enumerate}
     \item For a fixed $j \in J$, we fix an edge $e: \Delta^1 \to \Delta^n $ and for every $i \in I/\{j\}$, we fix vertices $x_i \in \Delta^n$, then the morphism 
     \begin{equation}
         (\prod_{i \in I/\{j\}}\Delta^0) \times \Delta^1 \cong\Delta^1 \xrightarrow{(e.\{x_i\})}      (\prod _{j \in J} \Delta^n)^{\op{op}} \times (\prod_{i \in J^c} \Delta^n) \xrightarrow{F_{\tau}} \op{Cat}_{\infty}
         \end{equation}
         admits a right adjoint
         \item For fixed $j \in J, i \in J^c$, we fixed two edges $e_1,e_2:\Delta^1 \to \Delta^n$, for $ k \in I/\{i,j\}$, we fix vertices $x_k$, then the commutative square given by the following morphism :
         \begin{equation}
          (\prod_{k \in I/\{i,j\}} \Delta^0) \times  \Delta^1 \times \Delta^1 \xrightarrow{(e_1,e_2,\{x_k\})}     (\prod _{j \in J} \Delta^n)^{\op{op}} \times (\prod_{i \in J^c} \Delta^n) \xrightarrow{F_{\tau}} \op{Cat}_{\infty}. 
         \end{equation}
         is right adjointable.
 \end{enumerate}
 The above conditions imply $F_{\tau}$ can be realized as a map :
 \begin{equation}
   F_{\tau}: (\prod_{i \in J^c} \Delta^n) \to \op{Fun}^{\op{RAd}}((\prod _{j \in J} \Delta^n)^{\op{op}},\op{Cat}_{\infty}).
 \end{equation}
 By \cref{adjsquarethm}, we get a morphism :
 \begin{equation}
     F_{\tau}^R : (\prod_{i \in J^c} \Delta^n) \to \op{Fun}^{\op{LAd}}((\prod _{j \in J} \Delta^n),\op{Cat}_{\infty}).
 \end{equation} 
 Using the Hom-product adjunction and applying the diagonal functor $\Delta^n \to \prod_{i\in I} \Delta^n$, we get the morphism :
 \begin{equation}
     \alpha'_{\tau} : \Delta^n \to (\prod_{i \in I} \Delta^n) \xrightarrow{F_{\tau}^R} \op{Cat}_{\infty} 
     \end{equation}
     In particular, for any $f:(n,\tau) \to (m,\tau')$ in $\Delta_{/\dd^*_{I,J}R}$, we have the following commutative diagram
     \begin{equation}\label{alphaprimefunctoriality}
         \begin{tikzcd}
             \Delta^n \arrow[r, "\alpha'_{\tau}"] \arrow[d,"f"] & \op{Cat}_{\infty} \\
             \Delta^m \arrow[ur, "\alpha'_{\tau'}"] & {}
             \end{tikzcd}
     \end{equation}
     In other words $\alpha'_{\tau'} \circ f \cong \alpha'_{\tau}$. Hence we can define the morphism $\alpha$ as follows :
     \begin{equation}
         \alpha(n,\tau) : \N(\tau) \to \op{Map}[\Delta^n,\op{Cat}_{\infty}] \quad;\quad (d:\tau \to \tau') \mapsto d \circ \alpha'(\tau')
     \end{equation}
     Notice that the target lies in $\op{Fun}^{\simeq}(\Delta^n,\op{Cat}_{\infty})$ beacuseof \cref{alphaprimefunctoriality}. 
 \end{enumerate}

 Thus by \cref{maintechnicalsimplicesthm} applied to $K'=\phi, K= \dd^*_{I,J}R$, we get a morphism 
 \begin{equation}
     F_J : \dd^*_{I,J}(R) \to \op{Cat}_{\infty}
     \end{equation}
which satisfied the conditons as required by the theroem.
\end{proof}

\begin{remark}
 As pointed out in \cite[Remark 9.4.11]{Robalothesis}, the above theorem of partial adjoints can be visualized as a property in terms of Gray-tensor product of scaled simplicial sets. The above simplicial sets appearing in the theorem as diagonals of $k$-simplicial sets are not $(\infty,1)$-categories but they possibly may admit a $(\infty,2)$-categorical structure. The details are not clear to us at the moment unfortunately but we suspect that above theorem has a generalization in the language of $(\infty,2)$-categories which may provide flexibility in dealing properties of six-functor formalisms.
    
\end{remark}
    \section{The $\infty$-category of correspondences.}
In this section, we recall the definitions and prove relevant statements concerning the $(\infty,1)$-category of correspondences. The category of correspondences serves as a source for our abstract six-functor formalisms in geometric setups. We also show that the $\infty$-category of correspondences admit a symmetric monoidal structure. Lastly, we also provide an alternate way of defining correspondences using the language of bisimplicial sets and describe the relation between this category and the simplicial sets $\dd^*_2\Ca^{\op{cart}}_{\E,\op{all}}$ that we consider in the context of $|\infty$-categorical compactifications.
    \subsection{Definitions and properties.}
    \begin{definition}\cite[Definition A.5.1]{padic6functorlucasmann}
        A \textit{geometric setup} is a pair $(\Ca,\E)$ where $\Ca$ is an $\infty$-category and $E$ is a homotopy class of edges in $\Ca$ satisfying
        \begin{enumerate}
            \item $E$ contains all isomorphism and is stable under compositions,

           \item Pullbacks of $E$ exist and remain in $E$.
        \end{enumerate}
    \end{definition}
    \begin{definition}\cite[Definition A.5.2]{padic6functorlucasmann}
        \begin{enumerate}
            \item Let $C(\Delta^n) \subset \Delta^n \times (\Delta^n)^{op}$ be the full-subcategory spanned by $([i],[j])$ with $i \le j$.
            \item An edge in $C(\Delta^n)$ is \textit{vertical} (resp. horizontal) if the projection to the second (first) factor is degenerate.
            \item A square in $C(\Delta^n)$ is \textit{exact} if it is both a pullback and a pushout square.
            \item For any simplicial set $K$, we define 
            \begin{equation*}
                C(K): = \op{colim}_{(n,\sigma) \in \Delta_{/K}}C(\Delta^n).
            \end{equation*}
            Thus $C$ can be visualized as an endofunctor on the category of simplicial sets. It admits a right adjoint 
            \begin{equation*}
                B: \sset \to \sset
            \end{equation*}
            defined by 
            \begin{equation*}
                K \mapsto B(K) =\{B(K)_n:= \op{Hom}_{\sset}(C(\Delta^n),K) \}.
            \end{equation*}
            \item Given a geometric setup $(\Ca,\E)$, define 
            \begin{equation*}
                \op{Corr}(\Ca)_{E,\op{all}} \subset B(\Ca)
            \end{equation*}
            to be the sub-simplicial set whose $n$ simplices are maps $C(\Delta^n) \to \Ca$ which sends 
            \begin{enumerate}
                \item vertical edges to $E$,
                \item exact squares to pullback squares.
            \end{enumerate}
        \end{enumerate}
    \end{definition}
    \begin{remark}
     For every $n \ge 1$, we have maps \begin{equation}
       \gamma_n : (\Delta^n)^{op} \to C(\Delta^n)  \quad ,\quad \gamma_n' : \Delta^n \to C(\Delta^n)
     \end{equation}
     given by the inclusion of the top row and the right most columns. These map induces maps :
     \begin{equation}\label{Cboundarymaps}
         \gamma_n : (\Lambda^n_0)^{op} \to C(\partial\Delta^n) \quad;\quad \gamma_n' : \Lambda^n_0 \to C(\Delta^n)
     \end{equation}
    \end{remark}
    \begin{remark}
Let $(\Ca,\E)$ be a geometric setup. Then the lower simplices of $\op{Corr}(\Ca)_{\E,\op{all}}$ look like as follows:
     \begin{itemize}
         \item The $0$-simplices are objects of $\Ca$.
         \item A $1$-simplex i.e an edge between $X_0$ and $X_1$ where $X_0,X_1 \in \Ca$ is a diagram of the form :
         \begin{equation*}
             \begin{tikzcd}
                 X_0 & X_{01}\arrow[l] \arrow[d,"f"] \\
                 {} & X_1
             \end{tikzcd}
         \end{equation*}
         where $f \in E$.
         \item A $2$-simplex in $\op{Corr}(\Ca)_{\E,\op{all}}$ looks like as follows:
         \begin{equation*}
             \begin{tikzcd}
                 X_0 & X_{01} \arrow[l] \arrow[d,"f"]  \arrow[dr, phantom, "\square"]  & X_{02} \arrow[l] \arrow[d,"g"] \\
                 {} & X_{11}  & X_{12} \arrow[d,"h"] \arrow[l] \\
                 {} & {} & X_{22}
             \end{tikzcd}
         \end{equation*}
         where $f,g,h\in E$ and $\square$ is a pullback square.
     \end{itemize}
     Notice that edges in $Corr(\Ca)_{\E,\op{all}}$ can be composed. In particular if we have two $1$ -simplices 
     \begin{equation*}
         \begin{tikzcd}
             X & X' \arrow[l] \arrow[d,"f"] \\
             {} & Y
         \end{tikzcd}
         ~~~~~~\text{and}~~~~~~
         \begin{tikzcd}
             Y & Y'\arrow[l] \arrow[d,"h"] \\
             {} & Z,
         \end{tikzcd}
     \end{equation*}
     we can compose these two $1$-simplices to form a $1$-simplex which is the outer roof of the following $2$-simplex
     \begin{equation*}
         \begin{tikzcd}
             X & X'\arrow[l] \arrow[d,"f"] & X' \times_{Y} Y' \arrow[l] \arrow[d,"g"] \\
             {} & Y & Y'\arrow[l] \arrow[d,"h"] \\
             {} & {} & Z.
         \end{tikzcd}
     \end{equation*}
     This is possible because we can take pullback of edges in $\E$ as $(\Ca,\E)$ is a geometric setup. 
    \end{remark}
    The above remark shows  that $\op{Corr}(\Ca)_{\E,\op{all}}$ admits lift along inclusions $\Lambda^2_1 \to \Delta^2$. This leads us to the following claim that $\op{Corr}(\Ca)_{\E,\op{all}}$ is an $\infty$-category:
    \begin{proposition}
        Let $(\Ca,\E)$ be a geometric setup. Then the simplicial set $\op{Corr}(\Ca)_{\E,\op{all}}$ is an $\infty$-category.
    \end{proposition}
    \begin{proof}
        By \cite[Corollary 2.3.2.2]{HTT}, it is enough to show that the map 
        \begin{equation*}
            \op{Fun}(\Delta^2,\op{Corr}(\Ca)_{\E,\op{all}}) \to \op{Fun} (\Lambda^2_1,\op{Corr}(\Ca)_{\E,\op{all}})
        \end{equation*}
        is a trivial fibration of simplicial sets.  Let $\K': = \op{Fun}(\Delta^2,\op{Corr}(\Ca)_{\E,\op{all}})$ and $\K:=\op{Fun} (\Lambda^2_1,\op{Corr}(\Ca)_{\E,\op{all}})$. We have obvious inclusions $\K' \subset \op{Fun}(C(\Delta^2),\Ca)$ and $\K \subset \op{Fun}(C(\Lambda^2_1),\Ca)$.  By definition of the functor $C$, we see that 
        \begin{equation*}
            C(\Delta^2) =C(\Lambda^2_1)^{\triangleleft}.
        \end{equation*}
        Now we notice the following observations :
        \begin{itemize}
            \item  The simplicial set $\K$ is the full subcategory of $\op{Fun}(C(\Lambda^2_1),\Ca)$ spanned by functors $F_0: C(\Lambda^2_1) \to \Ca$ which admit a limit i.e. $F_0$ can be extended to $C(\Delta^2)=C(\Lambda^2_1)^{\triangleleft} \to \Ca$. This is because as the vertical edges are spanned in $E$ and $(\Ca,\E)$ is a geometric setup.
            \item The simplicial set $\K'$ is the full subcategory of $\op{Fun}(C(\Delta^2),\Ca)$ spanned by functors $F: C(\Delta^2) \to \Ca$ which are right Kan extended from $F|_{C(\Lambda^2_1)}$. This is also true because $(\Ca,\E)$ is a geometric setup and exact squares get mapped to pullback squares in the definition of $\op{Corr}(\Ca)_{\E,\op{all}}$.
        \end{itemize}
   
    Thus one can apply the dual version of \cite[Proposition 4.3.2.15]{HTT} to $\K,\K'$ where $\Ca_0= C(\Lambda^2_1),\Ca = C(\Delta^2), \D = \Ca$ and $\D'=\Delta^0$. Thus we get that the map $\K' \to \K$ is a trivial fibration which completes our proof.
 
    \end{proof}

    Before upgrading the notion of correspondences in operadic version, we introduce some new notations on how the Correspondences behave with functor categories. This shall play a key role in extending six-functor formalisms from one geometric setup to another.
\begin{notation}
\begin{itemize}

    \item     Let $(\Ca,\E)$ be a geometric setup and $K$ be a simplicial set.  Then the functor category $(\op{Fun}(K,\Ca),\tilde{\E})$ is a geometric setup where $\tilde{\E}$ is set of edges $\Delta^1 \times K \to \Ca$ such that for all $k \in K$, the edge $\Delta^1 \times\{k\} \to \Ca$ is in $E$. For any set of edges $P$ stable under pullbacks and compositions in $\Ca$, let $\op{Fun}^P(K,\Ca)$ is full subcategory spanned by functors $f: K \to \Ca$ such that $f: K \to \Ca_P \to \Ca$.

\item Let $\op{Corr}^{\op{all-cart}}(\op{Fun}^{\E}(K,\Ca))_{\tilde{\E},\op{all}}$ be the full subcategory spanned by edges 
\begin{equation}\label{Corrfunctorialmap}
    \begin{tikzcd}
        f &  \arrow[l,"\sigma"] f' \arrow[d,"\tau"] \\
        {} & g
    \end{tikzcd}
\end{equation}
where $\tau \in \tilde{E}$ and for all $\gamma: \Delta^1 \to K$, the square $\sigma_{\gamma} : \Delta^1 \times \Delta^1 \to \Ca$ is cartesian. 
\item Using the notations from above, we have an canonical functor \[ \alpha_{K,\Ca,\E} : \op{Corr}^{\op{all-cart}}(\op{Fun}^{\E}(K,\Ca))_{\tilde{\E},\op{all}} \to \op{Fun}(K,\op{Corr}(\Ca)_{\E,\op{all}}) \]
As $C$ is a colimit preserving functor, it is enough to define the map for $K=\Delta^m$ where $m \ge 0$.
On the level of $n$-simplices, the map is defined as follows:

Let $\sigma_n$ be an $n$-simplex of $\op{Corr}^{\op{all-cart}}(\op{Fun}^{\E}(\Delta^m,\Ca))_{\tilde{\E},\op{all}}$, which is a map $C(\Delta^n) \times \Delta^m \to \Ca$. 
To obtain an $n$-simplex of RHS, we need to define a map $C(\Delta^n) \times C(\Delta^m) \to \Ca$. The map is defined by the following composition 
\begin{equation*}
    C(\Delta^n) \times C(\Delta^m) \to C(\Delta^n) \times \Delta^m \xrightarrow{\sigma_n}\Ca
\end{equation*}
the conditions of the the simplex $\sigma_n$ ensures that the morphism $C(\Delta^n) \times C(\Delta^m) \to \Ca$ can be upgraded to the a morphism $\Delta^n \times \Delta^m \to \op{Corr}(\Ca)$.

\item There is dual map 
   \begin{equation}\label{Corrfunctorialmapdual}
       \alpha'_{K,\Ca,\E} : \op{Corr}^{\op{\E-cart}}(\op{Fun}(K,\Ca))_{\tilde{E},\op{all}} \to \op{Fun}(K^{\op{op}},\op{Corr}(\Ca)_{\E,\op{all}}) .
   \end{equation}
\end{itemize}
\end{notation}
\begin{remark}
    The $\infty$-category $\op{Corr}(\Ca)_{\E,\op{all}}$ admits the following canonical maps :
    \begin{equation*}
        \Ca_{\E} \to \op{Corr}(\Ca)_{\E,\op{all}} \quad,\quad \Ca^{op} \to \op{Corr}(\Ca)_{\E,\op{all}}
    \end{equation*}
\end{remark}

\begin{proposition}\label{preservingcoproductsCor}
    Let $(\Ca,\E)$ be a geometric setup where $\Ca$ admits finite products. Then the map :
    \begin{equation*}
        \Ca^{op} \to \CrrCpEal
    \end{equation*}
    preserves finite coproducts.
\end{proposition}
\begin{proof}
    Let $X,Y \in \CrrCpEal$, we claim that the object $X \times Y$ with maps 
    \begin{equation*}
        \begin{tikzcd}
            X & \arrow[l] X \times Y \arrow[d,"\op{id}"] \\ {} & X \times Y
        \end{tikzcd}
        \quad ; \quad
        \begin{tikzcd}
            Y & \arrow[l] X \times Y \arrow[d,"\op{id}"] \\ {} & X \times Y
        \end{tikzcd}
    \end{equation*}
    makes $X \times Y$ as a coproduct of $X$ and $Y$.

    Let $F: \Delta^0 \coprod \Delta^0 \to \CrrCpEal$ be the diagram given by $X$ and $Y$. Let $F': (\Delta^0 \coprod \Delta^0)^{\triangleright} = \Lambda^2_0 \to \CrrCpEal$ be the diagram given by the two morphisms above. In order to show this is a colimit diagram, we need to show the following :
    \begin{itemize}
        \item         For~$n \ge 2$ and $H: (\Delta^0 \coprod \Delta^0) \boldsymbol{*}\partial\Delta^{n-1} \to \CrrCpEal$ be a diagram such that $H|_{[0]}=F'$. Then  $H~\text{extends to}~H':(\Delta^0 \coprod\Delta^0) \boldsymbol{*} \Delta^{n-1} \to \CrrCpEal$.
    \end{itemize}
The morphism $H$ is a map $(H_1,H_2):C(\Lambda^n_0) \coprod_{C(\partial\Delta^n)} C(\Lambda^n_0) \to \Ca$. Let $\{X_{i1}\}_{i=1}^m, \{X_{2i}\}_{i=1}^n$ be images of $H_1$ and $H_2$ along the first row of $C(\Lambda^n_0)$. Taking products $\{X_{i1} \times X_{i2}\}_{i=1}^n$, the morphism $H_1,H_2$ amalgamate to a morphism :
\begin{equation*}
    H_3: C(\Lambda^n_0)\coprod_K C(\Lambda^n_0) \to \Ca
\end{equation*}
where $K: = C(\Lambda^n_0)/\{(0,0)\}$. Here $H_3(0,i)$ for $i=1,2.\cdots,n$ is $X_{i1} \times X_{i2}$. For $n=2$ a pictorial description of $H_3$ is the following diagram :
\begin{equation}\label{nequals2diagram}
		\begin{tikzpicture}[baseline={(0,3)}, scale=2]
			
			\node (a1) at (0,2) {$ Y $};
			\node (a2) at (0,1.5) {$ X $};
			
			\node (b1) at (1,1.5) {$ X\times Y $};
			\node (b2) at (1,1) {$ X\times Y $};
			
			\node (c1) at (2,1.5) {$ X_{11}\times X_{22} $};
			\node (c2) at (2,1) {$  $};
			\node (c3) at (2,0.5) {$ Z $};
			
			\path[font=\scriptsize,>= angle 90]
			
			(b1) edge [->] node [above ] {$  $} (a1)
			(b1) edge [->] node [above ] {$  $} (a2)
			(b1) edge [double equal sign distance] node [above ] {$  $} (b2)
			
			(c1) edge [bend right=30,-, line width=1mm, white] node [] {} (a2)
			
			(c1) edge [bend right=30,->] node [above ] {$  $} (a1)
			(c1) edge [bend right=30,->] node [above ] {$  $} (a2)
			(c1) edge [bend left=30,->] node [above ] {$  $} (c3);
		\end{tikzpicture}
\end{equation}

%\begin{equation}
%    \begin{tikzcd}
%        Y  & {} & {} \\
%        X   &  X \times Y \arrow[l] \arrow[ul] \arrow[d,"\op{id}"] & X_{11} \times X_{12} \arrow[ll, bend right=70] \arrow[ull, bend right = 80]   \arrow[dd, bend left =40] \\
%        {} & X \times Y  & {} \\
%        {} & {} & Z
%        \end{tikzcd}
%\end{equation}
Restricting the morphism to the first row we get a map $h_3 : (\Lambda^n_0)^{op} \coprod_{(\partial\Delta^n)^{op}} (\Lambda^n_0)^{op}  \cong (\Delta^0 \coprod \Delta^0)\boldsymbol{*}\partial \Delta^n \to \Ca$ where restricting to $0$ vertex is the diagram $ X \leftarrow X \times Y \rightarrow Y$. As this is a limit diagram in $\Ca$, this  extends to a diagram \[h_3': (\Delta^0 \coprod \Delta^0)\boldsymbol{*} \Delta^n \to \Ca \] 

Amalgamating $h_3'$ with $H_3$, we see that $H_3$ extends to a morphism 
\begin{equation*}
    H_3' : C(\Lambda^n_0)' \coprod_{K}C(\Lambda^n_0)' \to \Ca
\end{equation*}
where the $'$ denotes including the first row of $C(\Delta^n)$ to the simplicial set. The rest of the proof follows from the following claim :
\begin{claim}
    The morphism $H_3'$ extends to a morphism \[H' : C(\Delta^n) \coprod_{C(\Delta^{n-1})} C(\Delta^n) \to \Ca.\]
\end{claim}
\begin{proof}[Proof of the claim]
We prove the claim in the following steps:
\begin{itemize}
    \item \textbf{Case n=2:} For $n=2$, we see that \cref{nequals2diagram}
 can be extended to the following diagram:
 
 \begin{equation*}
 	\begin{tikzpicture}[baseline={(0,3)}, scale=2]
 		
 		\node (a1) at (0,2) {$ Y $};
 		\node (a2) at (0,1.5) {$ X $};
 		
 		\node (b1) at (1,1.5) {$ X\times Y $};
 		\node (b2) at (1,1) {$ X\times Y $};
 		
 		\node (c1) at (2.25,1.5) {$ X_{11}\times X_{22} $};
 		\node (c2) at (2.25,1) {$ X_{11}\times X_{12} $};
 		\node (c3) at (2.25,0.25) {$ Z $};
 		
 		\path[font=\scriptsize,>= angle 90]
 		
 		(b1) edge [->] node [above ] {$  $} (a1)
 		(b1) edge [->] node [above ] {$  $} (a2)
 		(b1) edge [double equal sign distance] node [above ] {$  $} (b2)
 		(c1) edge [double equal sign distance] node [above ] {$  $} (c2)
 		(c2) edge [->] node [above ] {$  $} (b2)
 		(c1) edge [->] node [above ] {$  $} (b1)
 		
 		%ghost line for 3d-mensionality
 		(c1) edge [bend right=30,-, line width=1mm, white] node [] {} (a2)
 		
 		(c1) edge [bend right=30,->] node [above ] {$  $} (a1)
 		(c1) edge [bend right=30,->] node [above ] {$  $} (a2)
 		(c2) edge [->] node [above ] {$  $} (c3);
 	\end{tikzpicture}
 \end{equation*}
 
%\begin{equation}
%    \begin{tikzcd}
%        Y & {} & {}\\
%        X & X \times Y \arrow[ul] \arrow[l] \arrow[d,"\op{id}"] & X_{11} \times X_{12} \arrow[ull, bend right = 80]     \arrow[ll, bend right=70]  \arrow[d,"\op{id}"]  \arrow[l]\\
%        {} & X \times Y & X_{11} \times X_{12} \arrow[l] \arrow[d] \\
%        {} & {} & Z.
%    \end{tikzcd}
%\end{equation}
 which is the morphism $H'$ for $n=2$. For the rest of proof, hence we assume $n >2$.
 \item  We claim the  morphism $H_3'$ when restricted to $\widetilde{h_3}:=H_3':C(\partial\Delta^{n-1}) \subset K' \to \Ca $ can be extended to $C(\Delta^{n-1})$. Notice that the first row of $C(\partial\Delta^{n-1})$ is $\Lambda^n_0$. As we know that $H_3'$ is a morphism already is a map from each of $C(\Lambda^n_0)'$. Using \cref{isofibrationfunclift} twice, we see that $\widetilde{h_3}$ lifts to 
 \begin{equation*}
     \widetilde{h_3}' : C(\partial\Delta^{n-1})' \to \Ca
 \end{equation*}
 \item As $H_3'$ is derived from a morphism in $\CrrCpEal$, we see that the vertical edged between the first two rows are all identities (as the squares are cartesian). As equivalences are coCartesian morphisms from $\Ca \to \Delta^0$, we see that $\widetilde{h_3}'$ extends to a morphism 
 \begin{equation*}
     \widetilde{h_3}'' : C(\partial\Delta^{n-1)})'' \to \Ca
 \end{equation*}
 where $''$ means adjoining the top row and the right most vertical column.
 \item The morphism $\widetilde{h_3}''$ now extends to the whole simplex $C(\Delta^{n-1})$. The case $n=3$ can figured out pictorially. For higher $n$, we see that that commutativity of each exact squares follows from the existence of $C(\Delta^m)$ for $m <n-1$.
 \item The similar arguments also allow us to extend each of $C(\Lambda^n_0)''$ to $C(\Delta^n)$. This completes the claim that $H_3'$ extends to 
 \begin{equation*}
     H' : C(\Delta^n) \coprod_{C(\Delta^{n-1})} C(\Delta^n) \to \Ca.
 \end{equation*}
 
 \end{itemize}
    
\end{proof}
   \textit{Back to the main proof :} Once we have the existence of $H'$, the fact that all exact squares are pullback squares follows easily that lower simplices map to $\CrrCpEal$. Thus this shows $H'$ is indeed a morphism:
  \begin{equation*}
      H: (\Delta^0 \coprod \Delta^0)\boldsymbol{*} \Delta^{n-1} \to \CrrCpEal.
  \end{equation*}
  This completes the proof of the proposition.
\end{proof}
\begin{notation}
    let $(\Ca,\E)$ be a  geometric setup. We define a geometric setup on the $\infty$-category $(\Ca^{\op{op}})^{\coprod,\op{op}}$. We write an edge $f$ of $\Ccc$ in the form $\{Y_j\}_{1\le j \le n} \to \{X_i\}_{1 \le i \le m}$ lying over $\alpha : [m] \to [n]$. We define two sets of $\E^{+},\E^{-}$ as follows:
    \begin{itemize}
        \item $\E^+$ consists of $f$ such that the induced age $Y_{\alpha(i)} \to X_i$ belongs to $\E$ for every $i \in \alpha^{-1}(\langle n \rangle ^0)$,
        \item $\E^-$ is subset of $\E^+$ where $\alpha$ is degenerate.
    \end{itemize}
    We shall denote :
    \begin{equation*}
        \op{Corr}(\Ca)^{\otimes}_{\E,\op{all}} := \op{Corr}(\Ccc)_{\E^{-},\op{all}}
    \end{equation*}
\end{notation}

The following lemma address which edges in $\Ccc$ are coCartesian. This is essentially \cite[Lemma 6.1.4]{liu2017enhanced}. We provide the proof for the sake of completeness.

\begin{lemma}\label{coCaredginCor}
    Let $f$ be an edge in $\CrrCopEal$ of the form :
    \begin{equation*}
    \begin{tikzcd}
               \{X_j\}_{1\le j \le m}  & \arrow[l,"f_1"] \{Y_i\}_{1 \le i \le n} \arrow[d,"f_2"] \\
        {} & \{Z_i\}_{1 \le i \le n} 
    \end{tikzcd}
 \end{equation*}
 lying over $\alpha : \langle m \rangle \to \langle n \rangle$, then $f$ is p-coCartesian if and only if: 
 \begin{enumerate}
     \item for every $1 \le i \le n$, the morphism $Y_i \to Z_i$ is an isomorphism,
     \item for every $1 \le i \le n$, the morphism $Y_i \to  X_j$ where $\alpha(j)=i$ exhibits $Y_i \cong \prod_{j,\alpha(j)=i } X_j$. 
 \end{enumerate}
\end{lemma}
\begin{proof}
We prove the two implications as follows:
\begin{itemize}
    \item     \textbf{Proving the conditions provided that the morphism is coCartesian:} let $f$ be a morphism which is $p$-coCartesian. We prove the two conditions as follows:
    \begin{enumerate}
        \item \textbf{Proving $Y_j \to Z_j$ is an isomorphism:} As every morphism in $\CrrCopEal$ factorises as an active and inert. It is enough to show that if $f$ is of the following form:
        \begin{equation*}
            \begin{tikzcd}
                X &  Y \arrow[l] \arrow[d,"q"] \\
                {} & Z
            \end{tikzcd}
    \end{equation*}
    which is $p$-coCartesian , then $q$ is an isomorphism. Recall that a $f$ is $p$-coCartesian implies that the following diagram 
    \begin{equation*}
        \begin{tikzcd}
            \Delta^{\{0,1\}} \arrow[dr,"f"] \arrow[d] & {} \\
            \Lambda^n_0 \arrow[r,"\sigma"] \arrow[d] & \CrrCopEal \arrow[d,"p"] \\
            \Delta^n \arrow[r] \arrow[ur,dotted,"\sigma'"] & N(\op{Fin}_*)
         \end{tikzcd}
    \end{equation*}
    admits a solution. For $n=2$, and for two $1$-simplices given by $f$ and the morphism :
    \begin{equation*}
        \begin{tikzcd}
            X & Y \arrow[l] \arrow[d,"\op{id}"] \\
            {} & Y
        \end{tikzcd}
    \end{equation*}
    which form a morphism $\Lambda^2_0 \to \CrrCopEal$, the lifting problem gives us the following $2$ simplex $\tau$
    \begin{equation*}
        \begin{tikzcd}
            X & Y \arrow[d,"q"] \arrow[l] & Y \arrow[l,"\op{id}"] \arrow[d,"h_1"] \\
            {} & Z  & Y'\arrow[l,"g_1"] \arrow[d,"h_2"] \\
            {} & {} & Y
        \end{tikzcd}
    \end{equation*}

    We shall use this simplex while using the lifting problem for $n=3$. We construct a morphism $\sigma: \Lambda^3_0 \to \CrrCopEal$ given by $3$ 2-simplices : $\tau$: and the other two given as follows:
    \begin{equation*}
        \begin{tikzcd}
            X & Y \arrow[l] \arrow[d,"q"] & Y \arrow[l,"\op{id}"] \arrow[d,"q"]  \\
            {} & Z  & Z \arrow[l,"\op{id}"] \arrow[d,"\op{id}"] \\
            {} & {} & Z
        \end{tikzcd}
        \quad; \quad
        \begin{tikzcd}
            X & Y \arrow[l] \arrow[d,"\op{id}"] & Y \arrow[l,"\op{id}"] \arrow[d,"\op{id}"] \\
            {} & Y  & Y \arrow[l,"\op{id}"] \arrow[d,"q"] \\
            {} & {} & Z.
        \end{tikzcd}
    \end{equation*}
    The morphism $\sigma$ extends  to $\sigma'$ given by the $3$-simplex

    \begin{equation*}
        \begin{tikzcd}
            X & Y \arrow[l] \arrow[d,"q"] & Y \arrow[l,"\op{id}"] \arrow[d,"h_1"]  & Y \arrow[d,"q"] \arrow[l,"\op{id}"] \\
            {} & Z  & Y'\arrow[l,"g_1"] \arrow[d,"h_2"] & Z \arrow[l,"g_1"] \arrow[d,"q'"] \\
            {} & {} & Y & Y \arrow[l,"\op{id}"] \arrow[d,"q"] \\
            {} & {} & {} & Z
            {} & 
        \end{tikzcd}
    \end{equation*}
    The last vertical arrow shows that $q \circ q' = \op{id}_Z$ and $q' \circ q = \op{id}_Y$. This shows that $q$ is an isomorphism.
    \item In this case as the previous one, it is enough to show that if $f$ is of the form:
    \begin{equation*}
    \begin{tikzcd}
                \{X_j\}_{1 \le j \le m} & \arrow[d,"\op{id}"] Y \arrow[l,"\{f_j\}"] \\
                {} & Y
    \end{tikzcd}
    \end{equation*}
    is coCartesian, then $f_j$'s make $Y$ as a product of $\prod_{j=1}^m X_j$. We notice that $f$ lies in the image of the inclusion map :
    \[(\Ca^{\op{op}})^{\coprod} \to \CrrCopEal.\] For every $n \ge 2$, let us consider the lifting problem 
    
 \begin{equation*}
        \begin{tikzcd}
            \Delta^{\{0,1\}} \arrow[dr,"f"] \arrow[d] & {} \\
            \Lambda^n_0 \arrow[r,"\sigma"] \arrow[d] & (\Ca^{\op{op}})^{\coprod} \arrow[r] \arrow[d,"p'"] & \CrrCopEal \arrow[dl,"p"] \\
            \Delta^n \arrow[r,"\tau"] \arrow[urr,dotted,"\sigma'"] & N(\op{Fin}_*)
         \end{tikzcd}
    \end{equation*}
    where :
    \begin{enumerate}
        \item $\beta : \{X_j\}_{1 \le j \le m} \leftarrow Y \leftarrow Y_2 \cdots Y_n$.
        \item $\tau : \langle m \rangle \to \langle 1 \rangle \cdots \langle 1 \rangle$.
    \end{enumerate}
    Now as $\sigma'$ lifts to $\CrrCpEal$ and the squares are pullback squares. This shows that $\sigma'$ does factors via $\sigma'' : \Delta^n \to (\Ca^{\op{op}})^{\coprod}$. This implies that $f$ is already coCartesian in $(\Ca^{\op{op}})^{\coprod}$. By \cite[Remark 2.4.3.4]{HA}, this gives $Y$ is product of $X_j$'s.
    \end{enumerate}

    \item \textbf{Proving that $f$ is coCartesian:} As stated before, we would like to have a solution to the lifting problem :
    
    \begin{equation*}
        \begin{tikzcd}
            \Delta^{\{0,1\}} \arrow[dr,"f"] \arrow[d] & {} \\
            \Lambda^n_0 \arrow[r,"\sigma"] \arrow[d] & \CrrCopEal \arrow[d,"p"] \\
            \Delta^n \arrow[r,"\tau"] \arrow[ur,dotted,"\sigma'"] & N(\op{Fin}_*)
         \end{tikzcd}
    \end{equation*}
    where $\tau$ is given by $\langle k_0 \rangle \xrightarrow{\alpha} \langle k_1 \rangle \cdots \langle k_n \rangle$.
    As coCartesian morphisms are stable under compositions and the maps in $\op{Fin}_*$ admit a factorization system, we will show the lifting problem only for the following two cases :
    \begin{enumerate}
        \item \textbf{$\alpha$ is inert:} In that case $f$ is an equivalence in $\op{Corr}(\Ca)_{\E,\op{all}}$. Mimicking the arguments in \cite[Proposition 2.4.3.3]{HA}, we see that $f$ is $p$-coCartesian. 
    \item \textbf{$\alpha:\langle n \rangle \to \langle 1 \rangle$ active:} As $k_1 =1$, it is also enough to consider the case when $k_i=1$ for $i=2,\cdots ,n$. The morphism $\sigma$ induces for every $1\le j \le m$ the diagrams :
    \begin{equation*}
        \begin{tikzcd}
            \Delta^{\{0,1\}} \arrow[dr,"f_j"] \arrow[d] & {} \\
            \Delta^0  \boldsymbol{*} \partial\Delta^{n-1} \cong \Lambda^n_0 \arrow[r,"\sigma_j"] 
 &\op{Corr}(\Ca)_{\E,\op{all}}
        \end{tikzcd}
    \end{equation*}
    induce a diagram :
    \begin{equation*}
        h: \coprod_{j=1}^m\Delta^0 \boldsymbol{*}\partial\Delta^{n-1} \to \CrrCpEal
    \end{equation*}
    where $h|_0 : : (\coprod_{j=1}^m\Delta^0)^{\triangleleft} \to \op{Corr}(\Ca)_{\E,\op{all}}$. As $Y$ is a product of $X_j$'s and $\Ca$ admits finite products implies $\op{Corr}(\Ca)$ admits finite coproducts along the inclusion $\Ca^{op} \subset \CrrCpEal$ (\cref{preservingcoproductsCor}) . Thus $h|_0$ is a limit diagram and hence extends to 
    \begin{equation*}
        h' : \coprod_{j=1}^m \Delta^0 \boldsymbol{*} \Delta^{n-1} \to \CrrCpEal. 
    \end{equation*}
    The morphism $h'$ is indeed the desired  $\sigma' : \Delta^n \to \CrrCopEal$.
    
    \end{enumerate}
\end{itemize}

\end{proof}
\begin{proposition}
    Let $(\Ca,\E)$ be an $\infty$-category such that $\Ca$ admits finite products. Then 
    \begin{equation*}
        p : \CrrCopEal \to N(\op{Fin}_*)
    \end{equation*}
    is a coCartesian symmetric monoidal $\infty$-category whose underlying $\infty$-category is $\op{Corr}(\Ca)_{\E,\op{all}}$.
\end{proposition}
\begin{proof}
   From \cref{coCaredginCor}, we see that for any morphism in $N(\op{Fin}_*)$  and an element $(X_1,X_2,\cdots,X_n) \in \CrrCopEal$, we can construct a coCartesian edge. Thus the morphism $p$ is a coCartesian fibration. It follows from that the constructions that on every fiber $\langle n \rangle \in \op{Fin}_*$, we see that the fiber is $\CrrCpEal^n$. 
\end{proof}
\begin{remark}
\begin{enumerate}
    \item  

    The first map also admits a lifting in the operadic level, i.e we have a morphism of $\infty$-operads :
    \begin{equation*}
        (\Ca^{\op{op}})^{\coprod} \to \CrrCopEal
    \end{equation*}
    \item The maps $\alpha_{\Ca,\E,K}$ and $\alpha_{K,\Ca,\E}'$ from \cref{Corrfunctorialmap} and \cref{Corrfunctorialmapdual} also upgrade to the operadic level :
    \begin{equation}\label{Corroperadfunctorial}
        \alpha_{K,\Ca,\E}^{\otimes}: \op{Corr}^{\op{all-cart}}(\op{Fun}(K,\Ca))^{\otimes}_{\tilde{E},\op{all}}  \to \op{Fun}(K,\op{Corr}(\Ca)^{\otimes}_{\E,\op{all}})  \end{equation} and
        \begin{equation}\label{Corroperadfunctorialdual}
            \alpha_{K,\Ca,\E}^{'\otimes}: \op{Corr}^{\E-\op{cart}}(\op{Fun}(K,\Ca))^{\otimes}_{\tilde{E},\op{all}}  \to \op{Fun}(K^{\op{op}},\op{Corr}(\Ca)^{\otimes}_{\E,\op{all}}) 
        \end{equation}
        respectively.
\end{enumerate}
   
\end{remark}
\subsection{Correspondences and bisimplicial sets.}

In this section, we describe an alternative way of describing correspondences.  This uses the bisimplicial set $\Kpt^n$ defined in \cite[Definition 4.1.1]{chowdhury2024sixfunctorformalismsii}We define some general notion associated to bisimplicial sets which in geometric setups yields the $\infty$-category of correspondences. In addition to this, we prove a proposition which compares the mutlisimplical nerves ( \cite[Section 3]{chowdhury2024sixfunctorformalismsii}) and the $\infty$-category of correspondences. We advise the reader to look at the definition of various combinatorial sets from \cite[Section 4.1]{chowdhury2024sixfunctorformalismsii}.

\begin{definition}
    Let $X$ be a bisimplicial set, then we define the simplicial set $\dd^*_{2\nabla}X$ as follows:
    \begin{equation}
        \dd^*_{2\nabla}X_n := \op{Hom}_{\op{Set}_{2\Delta}}(\Kpt^n,X).
    \end{equation}
\end{definition}

\begin{example}
\begin{enumerate}
    \item $\dd^*_{2\nabla}\Kpt^n = \Copt^n$.
  \item   For $(\Ca,\E)$ a geometric setup, we have the simplicial set : 
    \begin{equation}
        \CrrCpEal \cong \dd^*_{2\nabla}(\op{op}^2_{\{2\}}\Ca^{\op{cart}}_{\E,\op{all}})
    \end{equation}
    \end{enumerate}
\end{example}

We have the following proposition :
\begin{proposition}\label{correspondencebisimplicialsetsequivalence.}
    Let $\D$ be an $\infty$-category. Let $h : \dd^*_2X \to \D$ be a map. Then the following dotted arrow exists making the diagram commute:
    \begin{equation}
        \begin{tikzcd}
            \dd^*_2 X \arrow[d,"r"] \arrow[r,"h"] & \D \\
            \dd^*_{2\nabla}X \arrow[ur,dotted,"h'"] & {}.
        \end{tikzcd}
    \end{equation}
\end{proposition}
\begin{remark}
    For $X= \Kpt^n$. The above proposition shows that we have solution of lifting problem for the morphism $ \bx^n\to \Copt^n$. This is true  by \cite[Proposition4.1.5]{chowdhury2024sixfunctorformalismsii}.
 \end{remark}
\begin{proof}
Let $\tau$ be an $n$-simplex of $\dd^*_{2\nabla}X$ given by $\Kpt^n \to X$. 

\begin{enumerate}

\item \textbf{Definition of $\N$}: We have the series of maps
\begin{equation}
    \alpha'(n,\tau): \Delta^0 \to \op{Fun}(\bx^n,\dd^*_2X) \to \op{Fun}(\bx^n,\D) 
\end{equation}
As $\bx^n \hookrightarrow \Copt^n$ is an inner anodyne(\cite[Proposition 4.1.5]{chowdhury2024sixfunctorformalismsii}), the restriction map  $\op{Fun}(\Copt^n,\D) \to \op{Fun}(\Copt^n,\D)$ is a trivial Kan fibration(\cite[Corollary 2.3.2.5]{HTT}), thus the morphism 
\begin{equation}
    \Delta^0 \times_{\op{Fun}(\bx^n,\D)} \op{Fun}(\Copt^n,\D) \to \Delta^0
\end{equation}
is a trivial Kan fibration which implies that the source is weakly contractible because $\Delta^0$ is contractible.
Set 
\begin{equation}
    \N(\tau): =   \Delta^0 \times_{\op{Fun}(\bx^n,\D)} \op{Fun}(\Copt^n,\D)
\end{equation}
\item \textbf{Construction of $\alpha$:} Let $\alpha(n,\tau)$ be the morphism :
\begin{equation}
    \alpha(n,\tau) : \N(\tau) \to \op{Fun}(\Copt^n,\D) \to \op{Fun}(\Copt^n,\D) \to \op{Fun}(\Delta^n,\D)
\end{equation}
Following the same arguments in Point 3 of proof of \cite[Theorem 4.2.1]{chowdhury2024sixfunctorformalismsii}, we see that $\alpha(n,\tau)$ can be realized as morphism :
\begin{equation}
\N(\tau) \to \op{Fun}^{\simeq}(\Delta^n,\D)
\end{equation}
By functoriality we get a natural transformation of functors :
\begin{equation}
    \alpha : \N \to \op{Map}[\dd^*_{2\nabla}X,\D]
\end{equation}
\item \textbf{Compatibility with $h$:}
For any $(n,\tau') \in \Delta_{/\dd^*_2X}$. Unravelling the definitions and applying $\dd^*_2$, this induces a morphism of the form
\begin{equation}
    \Delta^n \times \Delta^n \to \dd^*_2X
\end{equation}
This induces morphisms of the form $\overline{\tau''}:\Copt^n \to \dd^*_2X$ and $\tau'':\bx^n \to \dd^*_2X$ respectively.
By definition $r^*\alpha (n,\tau') = \alpha(n,r(\tau'))$.  The above discussions yields us that there is a natural choice of an element given by $\overline{\tau''} \in \N(\tau)$ such that $\alpha(n,r(\tau'))= \tau'$. Taking a compatible choice of $\overline{\tau''}$ gives us an element $\omega \in \Gamma(r^*\N)_0$ such that $\Gamma(r^*\alpha) = h'$.
\end{enumerate}
Thus by \cref{maintechnicalsimplicesthm}, we get a morphism $h : \dd^*_{2\nabla}X \to \D$ solving the lifting problem.
\end{proof}
  \section{Six-Functor Formalisms.}\label{App.A.2:_6FF}
 In this section, we recall the notion of $3$-functor and $6$-functor formalisms due to Mann. The main goal of this section is to show that under nice geometric setups with adjointability conditions along some subclasses of morphisms in $\E$, one can define abstract six-functor formalisms. This is essentially the content of \cite[Section 3.2]{liu2017enhanced}.

\subsection{Definitions}
\begin{definition}\cite[Definition A.5.6]{padic6functorlucasmann}
    Let $(\Ca,\E)$ be a geometric setup where $\Ca$ admits finite products. Then a \textit{pre-6-functor formalism/3-functor formalism} is a morphism of $\infty$-operads :
    \begin{equation*}
        \D_{(\Ca,\E)} : \CrrCopEal \to \op{Cat}^{\otimes}_{\infty}.
    \end{equation*}
    Given a pre-6-functor formalism, we introduce the following notations :
    \begin{enumerate}
        \item Restricting to the sub-operad $\Ccc$, we get a functor :
        \begin{equation*}
            \D^{*\otimes} : \Ca^{\op{op},\coprod} \to \op{Cat}^{\otimes}_{\infty}
        \end{equation*}
        This is equivalent to the functor 
        \begin{equation*}
            \D^{*\otimes} : \Ca^{\op{op}} \to \op{CAlg}(\op{Cat}_{\infty}).
        \end{equation*} 
        \item As $\D(X):=\D_{(\Ca,\E)}(X)$ is symmetric monoidal for every $X \in \Ca$, we get a tensor product structure :
        \begin{equation*}
            - \otimes - : \D(X) \times \D(X) \to \D(X).
        \end{equation*}
        \item Using the inclusion $\Ca_{\E} \to \CrrCpEal$, we get the following functor :
        \begin{equation*}
           \D_!: \Ca_{\E} \to \op{Cat}_{\infty}.
        \end{equation*}
    \end{enumerate}
\end{definition}

\begin{definition}\cite[Definition A.5.7]{padic6functorlucasmann}
    Let $(\Ca,\E)$ be a geometric setup such that $\Ca$ admits finite products. A \textit{6-functor formalism} is a pre-6-functor formalism $\D_{(\Ca,\E)}: \CrrCopEal \to \op{Cat}^{\otimes}_{\infty}$ such that 
    \begin{enumerate}
        \item For all $X\in \Ca$, $\D(X)$ is closed. The internal Hom  functor which is  the right adjoint of the tensor operation shall be denoted by  $\op{Hom}(-,-)$ .
        '\item For $f: X \to Y$ in $\Ca$, the morphism $f^*=\D^*(f)$ admits a right adjoint $f_* : \D(X) \to \D(Y)$. We denote by 
        \begin{equation*}
            \D_* : \Ca \to \op{Cat}_{\infty}
        \end{equation*}
        by the associated functor $X \mapsto \D(X), f \mapsto f_*$.
        \item For $f: X \to Y$ in $\Ca_{\E}$, the morphism $f_! : \D_!(f)$ admits a right adjoint $f^! : \D(Y) \to \D(X)$. We denote by
        \begin{equation*}
            \D^! : \Ca_{\E}^{op} \to \op{Cat}_{\infty}
        \end{equation*}
        the associated functor $X \mapsto \D(X), f \mapsto f^!$
    \end{enumerate}
\end{definition}
In the rest of this subsection, we describe how the a six functor formalism encodes properties like projection formula and base change.
\begin{enumerate}
    \item \textbf{Projection formula:} Let $f: X \to Y$ be a morphism in $E$, we consider the diagram :
    \begin{equation*}
        \begin{tikzcd}
            (X,Y) \arrow[r,"\sigma"] \arrow[d,"\tau"] & (Y,Y) \arrow[d,"\tau'"] \\
            X \arrow[r,"f"] & Y
        \end{tikzcd}
    \end{equation*}
    where :
    \begin{itemize}
        \item $\sigma':$\begin{equation*}
             \begin{tikzcd}
                (X,Y) & \arrow[l," \op{id}"] (X,Y) \arrow[d,"f_1"] \\
                {} & (Y,Y).
            \end{tikzcd}
        \end{equation*}
        where $f_1 = (f,\op{id})$.
        \item $\tau:$ \begin{equation*}
             \begin{tikzcd}
                (X,Y) & \arrow[l,"f_2"]  X \arrow[d,"\op{id}"] \\
               {} & X.
                \end{tikzcd}
        \end{equation*}
        where $f_2=(\op{id},f)$
        \item  $\tau':$\begin{equation*}
             \begin{tikzcd}
                (Y,Y) & Y \arrow[d,"\op{id}"] \arrow[l,"f_3"] \\
                {} & Y.
            \end{tikzcd}
        \end{equation*}
        where $f_3 = (\op{id},\op{id})$.
    \end{itemize}

This gives a morphism $\Delta^1 \times \Delta^1 \to \CrrCopEal$.  Applying $\D_{(\Ca,\E)}$, we get the following commutative square in $\op{Cat}_{\infty}$:
\begin{equation*}
    \begin{tikzcd}
        \D(X) \times \D(Y) \arrow[r,"f_! \times \op{id}"] \arrow[d, "\op{id} \otimes f^* "] & \D(Y) \times \D(Y) \arrow[d,"-\otimes-"]\\
        \D(X) \arrow[r,"f_!"] & \D(Y).
    \end{tikzcd}
\end{equation*}
which is the projection formula :
\begin{equation*}
    f_!((-) \otimes f^*(-)) \cong f_!(-) \otimes (-).
\end{equation*}
\item \textbf{Base change :} 
Let 
\begin{equation*}
    \begin{tikzcd}
        X' \arrow[d,"p'"] \arrow[r,"q'"] & X \arrow[d,"p"] \\
        Y' \arrow[r,"q"] & Y
    \end{tikzcd}
\end{equation*}
be a cartesian square with $p',p \in E$. We have a commutative square in $\CrrCopEal$ 
\begin{equation*}
    \begin{tikzcd}
        X \arrow[r] \arrow[d] & X' \arrow[d] \\
        Y\arrow[r,] & Y'
    \end{tikzcd}
\end{equation*}
which is comprised of two $2$-simplices
\begin{itemize}
    \item $\sigma_1:$
    \begin{equation*}
        \begin{tikzcd}
            X & \arrow[d,"p"] \arrow[l,"\op{id}"] X & X' \arrow[l,"q'"] \arrow[d,"p'"] \\
            {} & Y & Y'\arrow[l,"q"] \arrow[d,"\op{id}"] \\
            {} & {} & Y'.
         \end{tikzcd}
    \end{equation*}
    \item $\sigma_2':$
    \begin{equation*}
        \begin{tikzcd}
            X & X' \arrow[l,"q'"] \arrow[d,"\op{id}"] & X' \arrow[d,"\op{id}"] \arrow[l,"\op{id}"] \\
            {} & X' & X'\arrow[l,"\op{id}"]\arrow[d,"p'"]\\
            {} & {} & Y'.
        \end{tikzcd}
    \end{equation*}
    Applying $\D_{(\Ca,\E)}$ to the square, we get the commutative square :
    \begin{equation*}
        \begin{tikzcd}
            \D(X) \arrow[r,"q^{'*}"] \arrow[d,"p_!"] & \D(X') \arrow[d,"p'_!"]\\
            \D(Y) \arrow[r,"q^*"] & \D(Y')
        \end{tikzcd}
    \end{equation*}
    in $\op{Cat}_{\infty}$. Spelling this out, we get the base change equivalence $Ex^*_!$:
    \begin{equation}\label{App.:_Base_Change_Ex*_!}
    	q^*p_!\stackrel{Ex^*_!}{\simeq} p_!'q^{*'}
    \end{equation}
    
\end{itemize}
\end{enumerate}
\subsection{Construction of abstract 6FF.}

In this subsection, we prove the key result which roughly states on what conditions does a functor $\D^{*\otimes} : \Ca^{\op{op}} \to \op{CAlg}(\op{Cat}_{\infty})$ need to be upgraded to a 
 $3/6$-functor formalism.  The theorem is essential in reproving the construction of the Enhanced Operation Map due to Liu-Zheng (\cite{liu2017enhanced}).
 
 The conditions of the theorem are motivated from constructing the exceptional functors in the context of \'etale cohomology of schemes for separated of finite type of morphisms. In particular, using Nagata's compactification theorem, every morphism  $ f: X \to Y $ of quasi-compact and quasi-separated schemes which is separated and of finite type admits a  factorization by open immersion $j : X \to \bar{X}$ followed by a proper morphism $ p : \bar{X} \to Y$.  The exceptional functor $f_!$ is defined as $p_*j_{\#}$ where $j_{\#}$ is the left adjoint of $j^*$. \\

 In the case of geometric setups, we define analog of such decompositions of "open" and "proper" morphisms. Such decompositions will be called as "Nagata"-setup. We recall this definition from \cite{Dauser2024UniquenessOS}.
 \begin{definition}\cite[Definition 2.1]{Dauser2024UniquenessOS}
     A \textit{Nagata setup} $(\Ca,\E,\I,\P)$ is a geometric setup $(\Ca,\E)$ together with two subsets $\I,\P \subset \E$ such that :
     \begin{enumerate}
         \item $(\Ca,\I)$ and $(\Ca,\P)$ are geometric setups,
         \item Every morphism $f \in E$ admits a decomposition $f = \bar{f} \circ j$ where $j \in \I$ and $\bar{f} \in \P$.
         \item  Given $f : X \to Y$ in $\Ca$ and $g: Y \to Z$ in $\I$($\P$) then $f \in \I$($\P$) iff $g \circ f \in \I$($\P$).
         \item Every morphism $f \in \I \cap \P$ is $k$-truncated for some $k \ge -2$. 
         \end{enumerate}
 \end{definition}
 \begin{example}
     Let $\Ca=\op{Sch}$ be the category of Noetherian schemes of finite Krull dimension. Let $\E$ be morphisms which are separated of finite type. Choosing $\I$ as open immersions and $\P$ as proper morphisms, by Nagata compactification, we see that the tuple $(\Ca,\E,\I,\P)$ is a Nagata setup. 
 \end{example}
Now we proceed to state the main theorem which enables us to construct $3$/$6$-functor formalisms. This theorem  uses all the results from the previous articles.

\begin{theorem}\label{mainconstructiontheorem}
Let $(\Ca,\E,\I,\P)$ be a Nagata setup. Let
\begin{equation}
   \D^{*\otimes}: \Ca^{op} \to \op{CAlg}(\op{Pr}^L_{\op{st},\op{cl}})
\end{equation}
be a functor satisfying the following conditions : 
		\begin{enumerate}
			\item  For any  morphism in $\I$, $f^*$ has a left adjoint $f_{\#}$ such that: 
			\begin{enumerate}
				\item  ( $\I$-projection formula) For  any $E \in \D(Y)$ and $B \in \D(X)$, the natural map formed by adjunction
				\begin{equation}\label{smothpro}
					f_{\#}(E \otimes f^*(B)) \to f_{\#}E \otimes B \end{equation} is an equivalence.
				\item  ($\I$-base change) For a cartesian square
				\begin{equation}\label{8.1}
					\begin{tikzcd}
						X' \arrow[r,"f'"] \arrow[d,"g'"] & Y' \arrow[d,"g"] \\
						X \arrow[r,"f"] & Y
					\end{tikzcd}
				\end{equation}
				with $f \in \I$, the commutative square  
				
				\begin{equation}\label{8.2}
					\begin{tikzcd}
						\D(X') & \D(Y') \arrow[l,"\lbrace f '\rbrace^*"]  \\
						\D(X ) \arrow[u,"\lbrace g'\rbrace^*"] & \D(Y) \arrow[l,"f^*"] \arrow[u,"g^*"]
					\end{tikzcd}
				\end{equation}
				is horizontally left-adjointable i.e. there exists a commutative square
				\begin{equation}\label{8.3}
					\begin{tikzcd}
						\D(X') \arrow[r,"\lbrace f'\rbrace_{\#}"]& \D(Y') \\
						\D(X ) \arrow[u,"\lbrace g'\rbrace^*"] \arrow[r,"f_{\#}"] & \D(Y) \arrow[u,"g^*"]
					\end{tikzcd}
				\end{equation}
			\end{enumerate}
			\item For $f: Y \to X$ a  morphism  in $\P$ , $f^*$ admits a right adjoint functor $f_*$ with the following properties:
			\begin{enumerate}
				\item ($\P$-projection formula) For $E \in \D(Y)$ and $B \in \D(X)$, the natural map 
				\begin{equation}\label{proppro}
					f_*(E) \otimes B \to f_*(E \otimes f^*(B)) 
				\end{equation}
				is an equivalence.
				\item ($\P$ base change) For the cartesian square in \cref{8.1} with $f \in \P$, the induced pullback square \cref{8.2} is horizontally right adjointable. In other words, the square commutes
				\begin{equation}\label{8.4}
					\begin{tikzcd}
						\D(X') \arrow[r," f'_*"]& \D(Y') \\
						\D(X ) \arrow[u,"\lbrace g'\rbrace^*"] \arrow[r,"f_*"] & \D(Y) \arrow[u,"g^*"]
					\end{tikzcd}
				\end{equation}
			\end{enumerate} 
			\item (Support property) For a cartesian diagram in \cref{8.1} where $f \in \I$ and $g \in \P$,the commutative diagram in \cref{8.3} written as  square
			\begin{equation}\label{8.5}
				\begin{tikzcd}
					\D(X) \arrow[r,"f_{\#}"] \arrow[d,"\lbrace g' \rbrace ^*"] & \D(Y) \arrow[d,"g^*"] \\
					\D(X') \arrow[r,"\lbrace f' \rbrace_{\#}"] & \D(Y')
				\end{tikzcd}
			\end{equation} 
			is horizontally right adjointable, i.e. the square 
			
			\begin{equation}\label{8.6}
				\begin{tikzcd}
					\D(X') \arrow[r,"\lbrace f' \rbrace_{\#}"] \arrow[d,"g'_*"] & \D(Y') \arrow[d,"g_*"] \\
					\D(X) \arrow[r,"f_{\#}"] & \D(Y) 
				\end{tikzcd}
			\end{equation}
			commutes.
			
		\end{enumerate}
    Then $\D^{*\otimes}$ can be upgraded to a $3$-functor formalism : 
    \begin{equation}
        \D_{(\Ca,\E)} : \op{Corr}(\Ca)_{\E,\op{all}}^{\otimes} \to \op{Cat}_{\infty}^{\otimes}   \end{equation}
        such that for all $f \in \P$, $f_!=f_*$ and $f\in \I$, $f_!=f_{\#}$.
Morever if the $3$-functor formalism satisfies the additional assumptions :
\begin{enumerate}
    \item For every $X \in \Ca$, $\D(X)$ is closed.
    \item For every $f: X \to Y$ in $\Ca$, $f^*$ admits a right adjoint $f_*$.
    \item For every $f: X \to Y$ in $\P$, $f_*$ admits a right adjoint $f^!$.
    \end{enumerate}
Then the $3$-functor formalism is a $6$-functor formalism.         
	\end{theorem}
\begin{remark}
The above theorem is the content of \cite[Section 3.2]{liu2017enhanced}.
\end{remark}

    \begin{proof}
First we notice that once $\D_{(\Ca,\E})$ is a $3$-functor formalism, the additional conditions in the last part of theorem do provide the other three functors. Thus it boils down to proving the existence of $3$-functor formalism with desired properties as mentioned. Let us begin the steps of constructing the abstract 6FF formalism : 

\subsection*{\textbf{Step 1:}} We have the functor :
\begin{equation}
    \D^{\otimes*} : \Ca^{\op{op}} \to \op{CAlg}(\op{Pr}^L_{\op{st},\op{cl}}) \to \op{CAlg}(\op{Cat}_{\infty})
    \end{equation}

We have the following chain of compositions :
\begin{equation}
    \dd^*_{3,\{1,2,3\}}\Ca^{\op{cart}}_{\P,\O,\op{all}} \to \dd^*_{2,\{1,2\}}\Ca^{\op{cart}}_{\E,\op{all}} \to \dd^*_{1,\{1\}}\Ca^{\op{cart}}= \Ca^{op}
\end{equation}
Composing this chain of compositions with $\D^{\otimes*}$ we have the map : 

\begin{equation}
    \D_1 : \dd^*_{3,\{1,2,3\}}(\Ccc)^{\op{cart}}_{\P,\O,\op{all}} \to \op{CAlg}(\op{Cat}_{\infty})
\end{equation}
By universal property of $\op{CAlg}$ and $(-)^{\coprod}$, we have 
\begin{equation}
    \D_1 : \dd^*_{3,\{1,2,3\}}(\Ccc)^{\op{cart}}_{\P,\O,\op{all}} \to \op{Cat}_{\infty}
\end{equation}
Here we note that $\P,\O$ denotes sets of edges of the form $(X_1,X_2,..X_n) \to (Y_1,Y_2,...Y_n)$ where $X_i  \to Y_i$ are in $\P$ or $\O$ respectively.  

\subsection*{\textbf{Step 2:}}

We apply the theorem of partial adjoints in direction 1. In order to do this, we have the following lemma :

\begin{lemma}\label{projectionformulabasechangeadjunction}
    Let $\alpha : \langle m \rangle \to \langle n \rangle$ be a morphism in $\op{Fin}_*$. Consider the pullback square in $(\Ca^{\op{op}})^{\coprod,\op{op}}$
    \begin{equation}
        \begin{tikzcd}
            (X_1,X_2,...X_n) \arrow[d,"f'"]\arrow[r,"g"] & (Y_1,Y_2,...Y_m) \arrow[d,"f"]\\
            (X_1',X_2',..X_n') \arrow[r,"g'"] & (Y_1',Y_2'...Y_m')
        \end{tikzcd}
    \end{equation}
    which consists for every $1 \le i \le n$, the pullback squares in $\Ca$ :
    \begin{equation}
    \begin{tikzcd}\label{pullbacksquaresforkunneth}
        X_i\arrow[d,"f_i"] \arrow[r,"g_i"] & \prod_{\alpha(j)=i}Y_j \arrow[d,"\prod f_j"] \\
        X_i' \arrow[r,"g'_i"]  & \prod_{\alpha(j)=i} Y'_j 
        \end{tikzcd}
    \end{equation}
    where $f_i'$ amd $\prod f_j$ are in $\P$, then the square :
    \begin{equation}
        \begin{tikzcd}
            \prod_{i=1}^n \D(X_i) & \prod_{j=1}^m \D(Y_j) \arrow[l] \\
            \prod_{i=1}^n \D(X_i') \arrow[u] & \prod_{j=1}^m \D(Y_j') \arrow[l]\arrow[u]
        \end{tikzcd}
    \end{equation}
    is right adjointable on the vertical arrows using the pushforward map $(-)_*$.
\end{lemma}
\begin{proof}
    As map of finite sets decomposes over the images, it is safe to assume $n =1$ and $\alpha$ is active, we need to show that for pullback square of the form \cref{pullbacksquaresforkunneth}, we have 
    
    \begin{equation}
        \begin{tikzcd}[row sep=huge, column sep = huge]
            \D(X_i)  & \prod_{j=1}^m \D(Y_j) \arrow[l,"g_{1i}^* \otimes g_{2i}^* \cdots g_{ji}^*"] \\
            \D(X_i') \arrow[u,"f'^*_i"] & \prod_{j=1}^m \D(Y_j') \arrow[u,"\prod f_j^*"] \arrow[l,"g'^*_{1i} \otimes g'^*_{2i} \cdots g'^*_{ji}"]
        \end{tikzcd}
    \end{equation}
    where
    By induction, it suffices to assume for $m =2$, for $m=1$, it follows from $\P$-base change. Consider the following diagram :
  \begin{equation}
      \begin{tikzcd}
          \D(X_i) \arrow[d,"f_{i*}"] & \D(Y_1 \times Y_2) \arrow[l,"g_i^*"] \arrow[d,"(\prod f_{j})_*"] & \D(Y_1) \times \D(Y_2) \arrow[l,"- \boxtimes-"] \arrow[d," \prod f_{j*}"] \\
          \D(X_i') & \D(Y_1' \times Y_2') \arrow[l,"g_i^*"] & \D(Y_1') \times \D(Y_2') \arrow[l,"- \boxtimes -"]
      \end{tikzcd}
  \end{equation}
  The first square commutes via $\P$-base change, in order to show that the second square commute, as $(f_1 \times f_2)_* = (f_1 \times \op{id})_* \circ (\op{id} \times f_2)_*$, we assume $f_2 = \op{id}$. Consider the Cartesian square : 
  \begin{equation}
      \begin{tikzcd}
          Y_1 \times Y_2 \arrow[d,"f_1 \times \op{id}"] \arrow[r,"p_{Y_1}"] & Y_1 \arrow[d,"f_1"] \\
          Y_1' \times Y_2 \arrow[r,"p_{Y_1'}"] & Y_1'
      \end{tikzcd}
  \end{equation}
  The $\P$-base change gives us that : 
  \begin{equation}
      (f_1 \times \op{id})_* \circ p_{Y_1}^* \cong p_{Y_1'}^* \circ f_{1*}.
  \end{equation}

  For $M \in \D(Y_1), N \in \D(Y_2)$, we have
  \begin{align*}
       (f_1 \times \op{id})_*(M \boxtimes N) \\
      =(f_1 \times \op{id})_*( p_{Y_1}^*M \otimes p_{Y_2}^*N) \\
      = (f_1 \times \op{id})_*(p_{Y_1}^*M \otimes (f_1 \otimes \op{id})^*p_{Y_2}^*N) \\
      = (f_1 \times \op{id})_*(p_{Y_1}^*M) \otimes p_{Y_2}^*N  \quad (\P-\text{projection formula})\\
      = p_{Y_1}^*f_{1*}M \otimes p_{Y_2}^*N \\
      = f_{1*}M \boxtimes N
  \end{align*}
  This completes the proof of the second square is commutative and hence the bigger square is commuative which completes the proof of the lemma.
\end{proof}
The above lemma verifies the conditions of \cref{prtadj} applied to direction $1$ as we have $\P$-base change and $\P$-projection formula, this gives us the the map :
\begin{equation}
    \D_2 : \dd^*_{3,\{2,3\}}(\Ccc)^{\op{cart}}_{\P,\O,\op{all}} \to \op{Cat}_{\infty}
\end{equation}
The map $\D_2$ sends arrows in direction of $\P$ to $(-)_*$ and the remaining to pullbacks. 
\subsection*{\textbf{Step 3:}}

Now we apply partial adjoints (\cref{prtadj}) to the direction $2$. Using the lemma (\cref{projectionformulabasechangeadjunction}) instead of $\P$ applied to morphisms in $\O$, the $\O$-projection formula, $\O$-base change and support property verifies the condition of \cref{prtadj}. This then gives us the morphism :
\begin{equation}
    \D_3 : \dd^*_{3,\{3\}}(\Ccc)^{\op{cart}}_{\P,\O,\op{all}} \to \op{Cat}_{\infty}
\end{equation}
\subsection*{\textbf{Step 4:}}

Here we use the theory of $\infty$-categorical compactifications. By \cref{compthmalledge}, we can lift $\D_3$ to the functor
\begin{equation}
    \D_4 : \dd^*_{2,\{2\}}(\Ccc)^{\op{cart}}_{\E,\op{all}} \to \op{Cat}_{\infty}
\end{equation}

\subsection*{\textbf{Step 5:}} The last step is to identify this multisimplicial gadget with the $\infty$-category of correspondences.  This follows from \cref{correspondencebisimplicialsetsequivalence.}.
Thus this gives us the following map :
\begin{equation}
    \D_{\Ca,\E} : \op{Corr}(\Ca)^{\otimes}_{\E,\op{all}} \to \op{Cat}_{\infty}.
\end{equation}

By Proposition 2.4.1.7 of\cite{HA} which says that restriction map \[ \op{Fun}_{N(\op{Fin}_*)}(\CrrCopEal,\op{Cat}_{\infty}^{\otimes}) \to \op{Fun}^{\op{lax}}(\CrrCopEal,\op{Cat}_{\infty})\] is a trivial Kan fibration. Here the $\op{lax}$ superscript means all those functors $\pi: \CrrCopEal \to \op{Cat}_{\infty}$ such that for any object $(X_1,X_2,..X_n) \in \CrrCopEal$, we have $\pi(X_1,...X_n) = \prod_{i=1}^n \pi(X_i)$. This induces a lax-symmetric monoidal functor : 

\begin{equation}
    \D_{(\Ca,\E)} : \CrrCopEal \to \op{Cat}_{\infty}^{\otimes}.
\end{equation}
\end{proof}

\begin{remark}\label{alternateconstructionremarkMann}
As mentioned in the introduction, there is an ongoing work in progress by Mann, Heyer and Perutka of proving similar construction results of abstract six-functor formalisms (\cref{mainconstructiontheorem}) using the language of $(\infty,2)$-categories. As communicated to the author via email, the idea is to enhance the $(\infty,1)$-category of correspondences to $(\infty,2)$-level which shall encode morphisms lying in $\I$ and $\P$ on the level of $2$-morphisms. This enhancement allows to reformulate abstract six-functor formalisms and such constructions can be possibly carried out in the $(\infty,2)$-setup.
\end{remark}
\section{Extending Six-Functor Formalisms.}
In this subsection, we extend six functor formalism from smaller geometric setups to larger geometric setups. Similar extension results have been proven in DESCENT algorithms of Liu-Zheng (\cite[Section 4]{liu2017enhanced}) and also proved by Mann in his thesis (\cite[Lemma A.5.11,Proposition A.5.14]{padic6functorlucasmann}). We provide proofs of these results using the theory of localizations (in particular \cref{localizationcriterion}).

\subsection{Extension along nice geometric pairs.}
\begin{definition}
An inclusion of two marked $\infty$-categories $(\Ca,\S,\E) \subset (\Ca',\S',\E')$ is a \textit{nice geometric pair} if the following conditions hold :
    \begin{enumerate}
        \item Each of the four pairs $(\Ca,\S),(\Ca,\E),(\Ca,\S')$ and $(\Ca',\E')$ are geometric setups.
        \item $\S' \cap \Ca_1 =\S$.
        \item For  $X' \in \Ca'$, there exists a morphism $x : X \to X'$ called an \textit{atlas} such that $X \in \Ca$ and for every $Y \to X'$ where $Y \in \Ca$, the base change $Y \times_{X'} X \to Y$ lies in $S$.
        \item For every $f : X' \to Y'$ in $\E'$ and for every atlas $y : Y \to Y'$, the base change morphism $X' \times_{Y'}Y \to Y$ is in $\E.$
    \end{enumerate}
    \end{definition}
    \begin{example}
        Let $\Ca= \op{Sch}$ be the category of schemes and $\Ca'= \op{Algst}$ be the $(2,1)$-category of algebraic stacks. Considering $\S$ and $\S'$ as smooth surjections of schemes andalgebraic stacks respectively along with $\E$ and $\E'$ be representable morphisms of locally of finite type of schemes and algebraic stacks respectively, we see that $(\Ca,\S,\E) \subset (\Ca',\S',\E')$ is a nice geometric pair.
    \end{example}
\begin{proposition}\label{extendingsixfunctorCside}
   Let $(\Ca,\S,\E) \subset (\Ca',\S',\E')$ be a nice geometric pair. Let 
   \begin{equation}
       \D_{( \Ca,\E )} : \op{Corr}(\Ca)^{\otimes}_{\E,\op{all}} \to \op{Cat}^{\otimes}_{\infty}
   \end{equation}
   be a six-functor formalism such that the functor :
    \begin{equation*}
        \D^{*\otimes} : \C^{\op{op}} \to\op{CAlg}(\op{Pr}^L)
    \end{equation*}
    satisfies descent for $\S$-\v{C}ech-covers.
Then $\D_{(\Ca,\E)}$ can be extended to a six-functor formalism :
\begin{equation}
    \D_{(\Ca',\E')} : \op{Corr}(\Ca)^{\otimes}_{\E',\op{all}} \to \op{Cat}^{\otimes}_{\infty}.
\end{equation}
\end{proposition}
\begin{proof}
    We denote by $\op{Cov}(\Ca)$ the full subcategory of $\op{Fun}(N(\Delta)^{op},\Ca)$ spanned by \v{C}ech nerves of the atlases of elements $\Ca'$.An element of $\op{Cov}(\Ca)$ is given by $(x: X \to X',X')$. We have a canonical morphism 
    \begin{equation*}
        p : \CrrCovopEal \to \op{Corr}(\Ca')^{\otimes}_{\E',\op{all}}
    \end{equation*}
    Let $R$ be a collection of morphisms of the form: 
    \begin{equation*}
        \begin{tikzcd}
            (x_1 : X_1 \to X_1', X_1;\cdots; x_n : X_n \to X_n', X_n') \arrow[r,"f"] \arrow[d,"\op{id}"] &(y_1: Y_1 \to X_1', X_1';\cdots ; y_n : Y_n \to X_n', X_n') \\
            (x_1 : X_1 \to X_1', X_1;\cdots; x_n : X_n \to X_n', X_n') & {}.
        \end{tikzcd}
    \end{equation*} where $f=(f_i)_{i=1}^n$  and $f_i$ is a morphism of \v{C}ech nerves of $X_i$ for all $i$. The morphism $p$ sends $R$ to an equivalences. We have the following claim :
    \begin{claim}\label{claimforCorrCC'}
        The morphism $p$ is a localization of $\CrrCovopEal$ along $R$.
    \end{claim}
\begin{proof}[Proof of \cref{claimforCorrCC'}]
    It is enough to check the conditions of \cref{localizationcriterion}. First of all, the existence of atlases imply that $p$ is surjective on $n$-simplices. \\
    Second of all, if $(x: X \to X',X')$ and $(y: Y \to X',X')$ are two objects over $X'$, then we have the product of these two objects namely$( z: X \times _{X'} Y \to X',X')$.
Hence the conditions of \cref{localizationcriterion} are verified proving the claim.

\end{proof} 
We construct a morphism $\phi_{\Ca\Ca'} :\CrrCovopEal \to \op{Cat}_{\infty}^{\otimes}$ as follows:
\begin{itemize}
    \item The map $\alpha^{'\otimes}_{(\Delta_+)^{\op{op}},\Ca',\E'}$ (\cref{Corroperadfunctorialdual} provides a morphism :
    \begin{equation*}
        \phi_1 : \CrrCovopEal \xrightarrow{\alpha'_{(\Delta_+)^{\op{op}},\Ca',\E'}} \op{Fun}(N(\Delta_+),\op{Corr}(\Ca')^{\otimes}_{\E',\op{all}}) \xrightarrow{\op{res}} \op{Fun}(N(\Delta),\op{Corr}(\Ca)^{\otimes}_{\E,\op{all}})
    \end{equation*}
    where the second map is just restricting the augmented simplicial objects to the simplicial object. The simplicial object maps to $\CrrCpEal$ follows from the conditions in the proposition.
    \item The map $\D_{\Ca,\E}$ induces a functor :
    \begin{equation*}
       \phi_2: \op{Fun}(N(\Delta),\CrrCopEal) \to \op{Fun}(N(\Delta),\op{Cat}_{\infty}^{\otimes})
    \end{equation*}
    \item Taking limits of simplicial diagrams and using the theory of Kan extensions (\cite[Proposition 4.3.2.15]{HTT}, we have a limit functor :
    \begin{equation*}
        \phi_3 : \op{Fun}(N(\Delta),\op{Cat}_{\infty}^{\otimes}) \to \op{Cat}_{\infty}^{\otimes}
    \end{equation*}
    \item We define :
    \begin{equation*}
        \phi_{\Ca\Ca'}:= \phi_3 \circ \phi_2 \circ \phi_1.
    \end{equation*}
    In particular the morphism $\phi$ sends an object $(x_1: X_1 \to X_1', X_1';\cdots;x_n : X_n \to X_n',X_n')$ to the element $\Pi_{i=1}^n \op{lim}_{\bb \in \Delta} \D(X_{i,\bb})$.
\end{itemize}

We want to use that fact that as $p$ is a localization along $R$, then $\phi_{\Ca\Ca'}$ descends to a functor $\D_{\Ca'.\E'}$. For this we show the following claim :
\begin{claim}
    The functor $\phi_{\Ca\Ca'}$ sends $R$ to equivalences.
\end{claim}
\begin{proof}[Proof of the claim]
    Let $\X:= (X' \to X, X)$ and $\Y = (Y' \to X,X)$ be  two objects in $\op{Cov}(\Ca)$ and $f$ be a morphism in $\op{Cov}(\Ca)$ which induces an element of $R$. Let $\X \times \Y := (X' \times_{X} Y' \to X,X)$ be the product of $\X$ and $\Y$. We have the following diagram in $\op{Cov}(\Ca)$ :
    
    	\begin{center}
    	\begin{tikzpicture}[baseline={(0,1)}, scale=2]

    		\node (a) at (0,1) {$ \mc X\times \mc Y $};
    		\node (b) at (1, 1) {$ \mc Y $};
    		\node (c)  at (0,0.5) {$  \mc X $};
    		\node (e) at (0.2,0.75) {$  $};
    		\node (f) at (-0.75,1.5) {$ \mc X $};
    		\node (g) at (0.5,0.5) {$  $};

    		\path[font=\scriptsize,>= angle 90]

    		(a) edge [->] node [above ] {$ p_2 $} (b)
    		(a) edge [->] node [left] {$ p_1 $} (c)

    		(f) edge [bend right=-30,->] node [below] {$ f $} (b)
    		(f) edge [bend left=-30, double equal sign distance] node [below] {$  $} (c)
    		(f) edge [ ->] node [above] {$  $} (a);
    	\end{tikzpicture}
    \end{center}
    
%    \begin{equation}
%        \begin{tikzcd}
%            \X \arrow[dr] \arrow[ddr,"\op{id}", bend right=80] \arrow[drr,"f",bend left = 80] & {} & {} \\
%            {} & \X \times \Y \arrow[r,"p_1"] \arrow[d,"p_2"] & \Y \\
%            {} & \X & {}
%        \end{tikzcd}
%    \end{equation}
    \noindent which induces a diagram in $\CrrCovEopEal$ via the map:
    \[\op{Cov}(\Ca)^{\op{op}\coprod} \to \CrrCovopEal\]

As $\D^{*\otimes}$ satisfies descent along $S$-\v{C}ech covers, we see that $\phi_{\Ca\Ca'}$ sends $p_1$ and $p_2$ to 
 equivalences. Using the two out of three property of equivalences, we see that $\phi_{\Ca\Ca'}$ sends $f$ to equivalences.
 \end{proof}
 \textit{Back to the proposition:} As $p'$ is a localization along $R$ and $\phi_{\Ca\Ca'}$ sends $R$ to equivalences, we get that $\phi_{\Ca\Ca'}$ descends to a map :
 \begin{equation*}
     \D_{(\Ca',\E')} : \op{Corr}(\Ca')^{\otimes}_{\E',\op{all}} \to \op{Cat}^{\otimes}_{\infty}.
 \end{equation*}
 
 \end{proof}
 \subsection{Extension along exceptional pairs.}

 \begin{definition}
 An inclusion of $2$-marked $\infty$-categories $(\Ca,\S,\E) \subset (\Ca,\S,\E')$ is an \textit{exceptional pair} if the following conditions are satisfied :
    \begin{enumerate}
        \item The pairs $(\Ca,\S),(\Ca,\E),(\Ca,\E')$ are geometric setups.
        \item $\S \subset \E$.
        \item For every $f: X \to Y$ in $\E'$, there exists a morphism of augmented simplicial objects 
    \begin{equation*}
        f_{\bb} : X_{\bb} \to Y_{\bb}
        \end{equation*}
        where 
        \begin{itemize}
            \item $f_{-1}=f$
            \item $f_n \in E$ for $n \ge 0$,
            \item $X_{\bb} \to X$ and $Y_{\bb} \to Y$ are $\S$-hypercovers.
        \end{itemize}
        \end{enumerate}
        \end{definition}
        \begin{example}
         Let $\Ca$ be the category of schemes and $\E$ be the collection of morphisms which are separated and of finite type. Let $\S$ be the collection of Zariski covers. It follows from the definition that any morphism $f: X \to Y$ which is locally of finite type  admits a morphism of augmeneted simplicial objects $f_{\bb} : X_{\bb} \to Y_{\bb}$ where $X_{\bb}$ and $Y_{\bb}$ are Zariski hypercovers and $f_n$ for $n \ge 0$ is separated and of finite type. Thus letting $\E'$ to be the locally of finite type morphisms shows that $(\Ca,\S,\E) \subset (\Ca,\S',\E')$ is an exceptional pair.
        \end{example}
 \begin{proposition}\label{extendingsixfunctorEside}
   Let $(\Ca,\S,\E) \subset (\Ca,\S,\E')$ be an exceptional pair. Let 
   \begin{equation}
       \D_{( \Ca,\E )} : \op{Corr}(\Ca)^{\otimes}_{\E,\op{all}} \to \op{Cat}^{\otimes}_{\infty}
   \end{equation}
   be a six-functor formalism such that the functor
    \begin{equation*}
        \D_! : \C_E \to \op{Pr}^L
    \end{equation*}
    satisfies codescent for $\S$-hypercovers.
  
Then $\D_{(\Ca,\E)}$ can be extended to a six-functor formalism :
\begin{equation*}
    \D_{(\Ca,\E')} : \op{Corr}(\Ca)^{\otimes}_{\E',\op{all}} \to \op{Cat}^{\otimes}_{\infty}.
\end{equation*}
\end{proposition}
\begin{proof}
    The proof of this proposition follows the similar ideas of the proof of the \cref{extendingsixfunctorCside}.\\

    Let $\op{Cov}_{\E}(\Ca)$ be the full-subcategory of $\op{Fun}(N(\Delta)^{op},\Ca)$ spanned by simplicial objects which are $S$-hypercovers and spanned by $1$-simplices $f_{\bb} : X _{\bb} \to Y_{\bb}$ where if $f_{-1} \in \E'$, then $f_n \in \E$ for all $n \ge 0$.\\

    We have a canonical morphism :
    \begin{equation*}
        p_{\E} : \CrrCovEopEal \to \op{Corr}(\Ca)^{\otimes}_{\E',\op{all}}
    \end{equation*}
Let $R_{\E}$ be the collection of morphisms of the form :
\begin{equation*}
        \begin{tikzcd}
            (x_1 : X_1 \to X_1', X_1;\cdots; x_n : X_n \to X_n', X_n') \arrow[d,"f"] \arrow[r,"\op{id}"] &(x_1 : X_1 \to X_1', X_1;\cdots; x_n : X_n \to X_n', X_n')  \\
            (y_1: Y_1 \to X_1', X_1';\cdots ; y_n : Y_n \to X_n', X_n') & {}.
        \end{tikzcd}
    \end{equation*}
    where $f=(f_i)_{i=1}^n$ is a morphism of \v{C}ech nerves between $S$-hypercovers. Mimicking the ideas from the previous proposition and using \cref{localizationcriterion}, we see that $p_{\E}$ is a localization along $R_{\E}$. \\

    We construct a morphism 
    \begin{equation*}
        \phi_{\E\E'} : \CrrCovEopEal \to \op{Cat}^{\otimes}_{\infty}
    \end{equation*}
    as follows:
    \begin{itemize}
        \item The map $\alpha_{(\Delta_+)^{op},\Ca,\E'}$(\cref{Corrfunctorialmap}) induces a morphism :
        \begin{equation*}
            \phi_{1\E} : \CrrCovEopEal \to \op{Fun}(N(\Delta_+)^{\op{op}},\op{Corr}(\Ca)^{\otimes}_{\E',\op{all}}) \xrightarrow{\op{res}} \op{Fun}(N(\Delta)^{\op{op}},\CrrCopEal)
        \end{equation*}
        where the second map is restriction to $\CrrCopEal$ as morphisms in $\E'$ induce a morphism of augmented simplicial objects $f_{\bb} : X_{\bb} \to Y_{\bb}$ where $f_n \in \E$ for $n \ge 0$.
        \item The map $\D_{(\Ca,\E)}$ induces a functor :
        \begin{equation*}
            \phi_{2\E} : \op{Fun}(N(\Delta)^{\op{op}},\CrrCopEal) \to \op{Fun}(N(\Delta)^{\op{op}},\op{Cat}^{\otimes}_{\infty}).
        \end{equation*}
        \item Using the theory of Kan extensions (\cite[Proposition 4.3.2.15]{HTT}), we get a morphism :
        \begin{equation*}
            \phi_{3\E} : \op{Fun}(N(\Delta_+)^{\op{op}}, \op{Cat}^{\otimes}_{\infty}) \to \op{Cat}^{\otimes}_{\infty}.
        \end{equation*}
        \item We define :
        \begin{equation*}
            \phi_{\E\E'} := \phi_{3\E} \circ \phi_{2\E} \circ \phi_{1\E}
        \end{equation*}
        In particular, we see that $\phi_{\E\E'}(X_{\bb} \to X) \cong  \op{colim}_{\bb \in \Delta} \D_!(X_{\bb})$. In particular for any map $f: X \to Y$ in $\E'$ and $f_{\bb}: X_{\bb} \to Y_{\bb}$ a morphism of augmented simplicial objects where $f_n \in E$ , we see that \[\phi_{\E\E'}(f_{\bb}) \cong \op{colim}_{\bb \in \Delta}\D_!(f). \]
    \end{itemize}
    Again mimicking the ideas in the proof of \cref{extendingsixfunctorCside} and using that $\D_!$ has codescent along $\S$-hypercovers, we see that $\phi_{\E\E'}$ sends $R_{\E}$ to equivalences. As $p_{\E}$ is a localization and $\phi_{\E\E'}$ sends $R_{\E}$ to equivalences, $\phi_{\E\E'}$ descends to a morphism :
    \begin{equation*}
        \D_{(\Ca,\E')} : \op{Corr}(\Ca)^{\otimes}_{\E',\op{all}} \to \op{Cat}_{\infty}^{\otimes},
    \end{equation*}
\begin{remark}
\cref{extendingsixfunctorCside} and \cref{extendingsixfunctorEside} have been applied in various setups. Liu and Zheng have applied the DESCENT alogrithm (\cite[Section 4]{liu2017enhanced}) which is essentially a general reformulation of these statements in order to extend  six-functor formalism of \'etale cohomology from schemes to algebraic stacks. Another application of such statements can be found in joint work of the author with Alessandro D'Angelo (\cite{chowdhury2024nonrepresentablesixfunctorformalisms}) where extend the formalism of motivic homotopy theory from schemes to algebraic stacks. 
\end{remark}
    
\end{proof}

\begin{appendices}
\section{A criterion regarding localizations.}
In this section, we prove a proposition under what conditions a morphism of $\infty$-categories is a localization.
Let us recall the notion of Dwyer-Kan localizations.

\begin{definition}\cite[Definition 2.4.2]{land2021introduction}
Let $\Ca$ be an $\infty$-category and let $S \subset \Ca_1$  be a set of morphisms. A functor $\Ca \to \Ca[S^{-1}]$ is a \textit{Dwyer-Kan localization} of $\Ca$ along $S$ if for every auxiliary $\infty$-category $\D$ the functor :
\begin{equation*}
    \op{Fun}(\Ca[S^{-1}],\D) \to \op{Fun}(\Ca,\D)
\end{equation*}
    is fully-faithful and its essential image is all such functors $\Ca \to \D$ which sends $S$ to equivalences.
\end{definition}
\begin{remark}
    By \cite[Lemma 2.4.6]{land2021introduction}, localization exists along all morphisms of $\Ca$.
\end{remark}
\begin{proposition}\label{localizationcriterion}
    Let $\Ca$ be an $\infty$-category and $R$ be a set of morphisms in $\Ca$. Let $p : \Ca \to \D$ which sends $R$ to equivalences. Suppose we have the following conditions: 

\begin{itemize}
    \item $p$ is surjective on $n$-simplices for $n \ge 0$.
    \item For every $d \in \D$, the $\infty$-category $\Ca_d$ admits products. In particular $x,y \in \Ca_d$, the projection maps $x \times y \to x , x \times y \to y$ are in $R$
\end{itemize}

    Then $p$ is a Dwyer-Kan localization along $R$.
\end{proposition}
\begin{proof}
    As $p$ sends $R$ to equivalences, it factorizes to a map \[p':\Ca[R^{-1}] \to \D. \]
The goal is to show $p'$ is a categorical equivalence. As trivial fibrations are categorical equivalences (\cite[Proposition 2.2.12]{Land_introductionQC}),we show that $p'$ is a trivial fibration i.e for $m \ge 0$, the diagram below admits a solution:
\begin{equation*}
    \begin{tikzcd}
        \partial\Delta^m \arrow[r,"\sigma"] \arrow[d,hookrightarrow] & \Ca[R^{-1}] \arrow[d,"p'"] \\
        \Delta^m \arrow[ur,dotted,] \arrow[r,"\tau"] & \D
    \end{tikzcd}
\end{equation*}
The case $m =0$ follows as map $p'$ is surjective on objects. Thus we assume $m >0$.
We know that $\tau$ lifts to $\tau' : \Delta^n \to \Ca$ which even restricts to $ \sigma':\partial\Delta^n \to \Ca$. Let $X_0,X_1,....X_n$ be the vertices of $\sigma'$ and $Y_0,Y_1,Y_2,...Y_n$ be vertices of $\sigma$. By condition of the proposition, we have vertices $Z_i := X _i \times Y_i$ with maps $p_i: Z_i \to X_i, q_i : Z_i 
\to Y_i$ are in $R$.
We make the following two observations :
\begin{enumerate}
\item The vertices $X_i,Z_i$ when realized in $\Ca[R^{-1}]$ along with $\sigma'$ amalgamate to define a morphism 
\begin{equation}
    h_1 ; \Delta^1 \times (\coprod_{j=0}^m \Delta^0) \times_{[1] \times (\coprod_{j=0}^m \Delta^0)} \Delta^m \to \Ca[R^{-1}]
\end{equation}
where for all $0\le k \le n$, $h_1|_{\Delta^1 \times [k]}$ is an equivalence.
\item The vertices $Y_i,Z_i$ along with $\sigma$ amalgamate to define a morphism :
\begin{equation}
    h_2 : \Delta^1 \times(\coprod_{j=0}^m \Delta^0) \times_{[1] \times (\coprod_{j=0}^m\Delta^0)} \partial\Delta^m \to \Ca[R^{-1}]
\end{equation}
where for all $0\le k \le n$, $h_2|_{\Delta^1 \times [k]}$ is an equivalence

 \end{enumerate}
The proof follows from the following claim : 
\begin{claim}\label{isofibrationfunclift}
		Let $\E$ be an $\infty$-category and $n \ge 1$.  Let \[h: \Delta^1 \times \partial\Delta^n \coprod_{\{0\} \times \partial\Delta^n} \{0\} \times\Delta^n \to \E \] be a morphism such that $h|_{\Delta^1 \times [k]}: \Delta^1 \to \D $ is an equivalence for all $ 0 \le k \le n$.  Then there exists a morphism $h': \Delta^n \times \Delta^1 \to \D$ such that the diagram 
		\begin{equation*}
			\begin{tikzcd}
				\Delta^1 \times \partial\Delta^n \coprod_{\{0\} \times \partial\Delta^n} \{0\} \times\Delta^n \arrow[r,"f"] \arrow[d,hookrightarrow] & \E  \\
				\Delta^n \times \Delta^1 \arrow[ur,"h'"] & {}
			\end{tikzcd}
		\end{equation*}
		commutes.
	\end{claim}
	\begin{proof}[Proof of claim]
		The morphism $h$ gives us the following commutative diagram 
		\begin{equation*}
			\begin{tikzcd}
				\{0\} \arrow[r] \arrow[d, hookrightarrow] & \op{Fun}(\Delta^n,\E) \arrow[d,"p"] \\
				\Delta^1 \arrow[r,"g'"] & \op{Fun}(\partial\Delta^n,\E)
			\end{tikzcd}
		\end{equation*}
		where $g'$ is an equivalence. As $i: \partial \Delta^n \to \Delta^n$ is bijective on $0$ simplices applying \cite[Proposition 2.2.5]{Land_introductionQC} gives us that $p$ is an isofibration. Thus there exists a morphism $h': \Delta^1 \to \op{Fun}(\Delta^n,\E)$ which extends $h$. This completes the proof.
	\end{proof}
 \textit{Back to proving the proposition:} The idea is to apply the  \cref{isofibrationfunclift} two times as follows:
 \begin{enumerate}
     \item Applying proposiion to $h_1$ starting from all $n=11$-subsimplices to $n=m$ extends $h_1$ to the following morphism : 
\begin{equation}
    h_1': \Delta^1 \times \Delta^m \to \Ca[R^{-1}]
\end{equation}
In particular the $Z_i$'s are now vetices of an $m$-simplex $\sigma'': h_1'|_{\Delta^1 \times [0]} : \Delta^m \to \Ca[R^{-1}]$.
     \item Applying the proposition with the amalgamation of $h_2$ and $\sigma''$ extends $h_2$ to the following morphism: 
     \begin{equation}
         h_2' : \Delta^1 \times \Delta^m \to \Ca[R^{-1}]
     \end{equation}
     In particular $h_2'|_{[1] \times \Delta^m}: \Delta^m \to \Ca[R^{-1}]$ extends $\sigma$  solving the lifting problem. 
 \end{enumerate}
	This completes the proof that $p'$ is a trivial fibration and the proof of the proposition is complete. 
\end{proof}

\end{appendices}

 \bibliographystyle{abbrv}

\end{document}